\documentclass[11pt]{article}

\usepackage{amsmath}
\usepackage{amsthm}
\usepackage{graphicx}
\usepackage{color}
\usepackage{amsfonts}
\usepackage{amssymb}
\usepackage{psfrag}
\usepackage{wrapfig}
\usepackage{amscd}
\usepackage{url}

\newcommand{\re}{\mathbb{R}}
\newcommand{\cpx}{\mathbb{C}}

\newcommand{\N}{\mathbb{N}}
\newcommand{\Q}{\mathbb{Q}}
\renewcommand{\P}{\mathbb{P}}

\newcommand{\diag}{\mbox{diag}}

\newcommand{\half}{\frac{1}{2}}
\newcommand{\lmd}{\lambda}

\newcommand{\eps}{\epsilon}

\def\af{\alpha}

\def\rank{\mbox{rank}}

\newcommand{\sig}{\sigma}

\newcommand{\reff}[1]{(\ref{#1})}
\newcommand{\pt}{\partial}
\newcommand{\prm}{\prime}

\newcommand{\mc}[1]{\mathcal{#1}}

\newcommand{\bdes}{\begin{description}}
\newcommand{\edes}{\end{description}}

\newcommand{\bal}{\begin{align}}
\newcommand{\eal}{\end{align}}

\newcommand{\bnum}{\begin{enumerate}}
\newcommand{\enum}{\end{enumerate}}

\newcommand{\bit}{\begin{itemize}}
\newcommand{\eit}{\end{itemize}}

\newcommand{\bea}{\begin{eqnarray}}
\newcommand{\eea}{\end{eqnarray}}
\newcommand{\be}{\begin{equation}}
\newcommand{\ee}{\end{equation}}

\newcommand{\baray}{\begin{array}}
\newcommand{\earay}{\end{array}}

\newcommand{\bsry}{\begin{subarray}}
\newcommand{\esry}{\end{subarray}}

\newcommand{\bca}{\begin{cases}}
\newcommand{\eca}{\end{cases}}

\newcommand{\bcen}{\begin{center}}
\newcommand{\ecen}{\end{center}}

\newcommand{\bbm}{\begin{bmatrix}}
\newcommand{\ebm}{\end{bmatrix}}

\newcommand{\bmx}{\begin{matrix}}
\newcommand{\emx}{\end{matrix}}

\newcommand{\bpm}{\begin{pmatrix}}
\newcommand{\epm}{\end{pmatrix}}

\newcommand{\btab}{\begin{tabular}}
\newcommand{\etab}{\end{tabular}}

\newtheorem{theorem}{Theorem}[section]

\newtheorem{prop}[theorem]{Proposition}
\newtheorem{lem}[theorem]{Lemma}

\newtheorem{cor}[theorem]{Corollary}

\theoremstyle{definition}

\newtheorem{exm}[theorem]{Example}

\usepackage[margin=1.2in]{geometry}

\begin{document}

\title{ Discriminants and Nonnegative Polynomials
\author{Jiawang Nie\footnote{Department of Mathematics,
University of California, 9500 Gilman Drive, La Jolla, CA 92093.
Email: njw@math.ucsd.edu. The research was partially supported by NSF grants
DMS-0757212, DMS-0844775 and Hellman Foundation Fellowship.}}
\date{April 22, 2010}
}

\maketitle

\begin{abstract}
For a semialgebraic set $K$ in $\re^n$, let
$
P_d(K)=\left\{f\in \re[x]_{\leq d}: f(u) \geq 0 \,\forall\, u \in K\right\}
$
be the cone of polynomials in $x\in \re^n$ of degrees at most $d$ that are nonnegative on $K$.
This paper studies the geometry of its boundary $\pt P_d(K)$.
When $K=\re^n$ and $d$ is even, we show that its boundary $\pt P_d(K)$
lies on the irreducible hypersurface defined by the discriminant $\Delta(f)$ of $f$.
When $K=\{x\in\re^n:g_1(x)=\cdots=g_m(x)=0\}$ is a real algebraic variety,
we show that $\pt P_d(K)$ lies on the hypersurface defined
by the discriminant $\Delta(f,g_1,\ldots,g_m)$ of $f,g_1,\ldots,g_m$.
When $K$ is a general semialgebraic set, we show that $\pt P_d(K)$
lies on a union of hypersurfaces defined by the discriminantal equations.
Explicit formulae for the degrees of these hypersurfaces and discriminants are given.
We also prove that typically $P_d(K)$ does not have a barrier of type $-\log \varphi(f)$
when $\varphi(f)$ is required to be a polynomial,
but such a barrier exits if $\varphi(f)$ is allowed to be semialgebraic.
Some illustrating examples are shown.
\end{abstract}



\section{Introduction}

Let $K$ be a semialgebraic set in $\re^n$, and $P_d(K)$ be the cone of
multivariate polynomials in $x\in \re^n$ that are nonnegative on $K$
and have degrees at most $d$, that is,
\[
P_d(K)=\left\{f\in \re[x]_{\leq d}: f(u) \geq 0 \,\forall\, u \in K\right\}.
\]
A very natural question is what is the boundary of $P_d(K)$?
What kind of equation does it satisfy?
Can we find a nice barrier function for $P_d(K)$?
This paper discusses these issues.

A polynomial $f(x)$ in $x\in \re^n$ is said to be nonnegative or positive semidefinite (psd)
on $K$ if the evaluation $f(u)\geq 0$ for every $u\in K$.
When $K=\re^n$ and $d$ is even, an $f(x)\in P_d(\re^n)$
is called a nonnegative polynomial or psd polynomial.
When $K=\re_+^n$ is the nonnegative orthant, an $f(x) \in P_d(\re_+^n)$
is called a co-positive polynomial.
Typically, it is quite difficult to check the membership of the cone $P_d(K)$.
In case of $K=\re^n$, for any even $d >2$,
it is NP-hard to check the membership of $P_d(\re^n)$.
In case of $K=\re_+^n$, for any $d>1$,
it is NP-hard to check the membership of $P_d(\re_+^n)$.
In practical applications, people usually do not check the membership of $P_d(K)$ directly,
and instead check sufficient conditions like sum of square (SOS) type representations
(a polynomial is SOS if it is a finite summation of squares of other polynomials).
There is much work on applying SOS type certificates to approximate the cone $P_d(K)$.
We refer to \cite{Las01,NDS,Par00,ParStu,Put,Smg}.
However, there is relatively few work on studying the cone $P_d(K)$ and its boundary $\pt P_d(K)$ directly.
The geometric properties of $\pt P_d(K)$ are known very little.

When $K=\re^n$ and $d=2$, $P_2(\re^n)$ reduces to the cone of positive semidefinite matrices,
because a quadratic polynomial $f(x)$ is nonnegative everywhere
if and only if its associated symmetric matrix
$A\succeq 0$(positive semidefinite).
The boundary of $P_2(\re^n)$ consists of $f$ whose corresponding $A$
is positive semidefinite and singular, which lies on
the irreducible determinantal hypersurface $\det(A) = 0$.
Its degree is equal to the length of matrix $A$.
A typical barrier function for $P_2(\re^n)$ is $- \log \det(A).$
Note that $\det(A)$ is a polynomial in the coefficients of $f(x)$.
Do we have a similar result for $P_d(K)$ when $K \ne \re^n$ or $d > 2$?
Clearly, when $K=\re^n$ and $d>2$, we need to generalize
the definition of determinants for quadratic polynomials to higher degree polynomials.
There has been classical work in this area like \cite{GKZ}.
The ``determinants" for polynomials of degree $3$ or higher are called {\it discriminants}.
The discriminant $\Delta(f)$ of a single homogeneous polynomial (also called form) $f(x)$
is defined such that $\Delta(f)=0$ if and only if $f(x)$ has a nonzero critical point.
For a general semialgebraic set $K$, to study $\pt P_d(K)$,
we need to define the discriminant $\Delta(f_0,\ldots,f_m)$
of several polynomials $f_0,\ldots, f_m$.
As we will see in this paper, the discriminant plays a fundamental role in studying $P_d(K)$.

Recently, there are arising interests in the new area of convex algebraic geometry.
The geometry of convex (also including nonconvex) optimization problems
would be studied by using algebraic methods.
There is much work in this field, like maximum likelihood estimation \cite{CHKS},
$k$-ellipse \cite{NPS}, semidefinite programming \cite{NRS,RvB},
matrix cubes \cite{NiSt09}, polynomial optimization \cite{NR09},
statistical models and matrix completion \cite{StuUhl},
convex hulls \cite{Hen09,RaSt09,SSS}, theta bodies \cite{GPT}.
In this paper, we study the geometry of the cone $P_d(K)$
by using algebraic methods, and find its new properties.

\bigskip
\noindent
{\bf Contributions} \quad
The cone $P_d(K)$ is a semialgebraic set,
and its boundary $\pt P_d(K)$ is a hypersurface defined by
a certain polynomial equation. To study this hypersurface,
we need to define the discriminant $\Delta(f_0,\ldots,f_m)$
for several forms $f_0,\ldots, f_m$,
which satisfies $\Delta(f_0,\ldots,f_m)=0$
if and only if $f_0(x)=\cdots=f_m(x)=0$ has a nonzero singular solution.
This will be shown in Section~\ref{sec:dis-sevral}.
When $K=\re^n$ and $d>2$ is even, we prove that $\pt P_d(\re^n)$
lies on the irreducible discriminantal hypersurface $\Delta(f)=0$,
which will be shown in Section~\ref{sec:psd}.
When $K = \{x\in\re^n: g_1(x)=\cdots=g_m(x)=0\}$ is a real algebraic variety,
we show that $\pt P_d(K)$ lies on the discriminantal hypersurface $\Delta(f_0,\ldots,f_m)=0$,
which will be shown in Section~\ref{sec:real-var}.
When $K$ is a general semialgebraic set, we show that $\pt P_d(K)$
lies on a union of several discriminantal hypersurfaces,
which will be shown in Section~\ref{sec:semi-alg}.
Explicit formulae for the degrees of these hypersurfaces will also be shown.
Generally, we show that $P_d(K)$ does not have a barrier of type $-\log \varphi(f)$
when $\varphi(f)$ is required to be a polynomial,
but such a barrier exits if $\varphi(f)$ is allowed to be semialgebraic.
For the convenience of readers, we include some preliminaries about
elementary algebraic geometry, discriminants and resultants.
This will be shown in Section~\ref{sec:pre}.

\section{Some preliminaries}  \label{sec:pre}
\setcounter{equation}{0}

\subsection{Notations}
The symbol $\N$ (resp., $\re$) denotes the set of nonnegative integers (resp., real numbers),
and $\re_+^n$ denotes the nonnegative orthant of $\re^n$.
For integer $n>0$, $[n]$ denotes the set $\{1,\ldots,n\}$.
For $x \in \re^n$, $x_i$ denotes the $i$-th component of $x$,
that is, $x=(x_1,\ldots,x_n)$,
and $\tilde{x}$ denotes  $(x_0,x_1,\ldots,x_n)$.
For $\af \in \N^n$, denote $|\af| = \af_1 + \cdots + \af_n$.
For $x \in \re^n$ and $\af \in \N^n$,
$x^\af$ denotes $x_1^{\af_1}\cdots x_n^{\af_n}$.
The $[x^d]$ denotes the column vector of all monomials of degree $d$, i.e.,
$
[x^d]^T = [\, x_1^d \quad x_1^{d-1}x_2 \quad \cdots \cdots \quad x_n^d \,].
$
The symbol $\re[x] = \re[x_1,\ldots,x_n]$ (resp. $\cpx[x] = \cpx[x_1,\ldots,x_n]$)
denotes the ring of polynomials in $(x_1,\ldots,x_n)$ with real (resp. complex) coefficients;
$\re[\tilde{x}] = \re[x_0,x_1,\ldots,x_n]$ and $\cpx[\tilde{x}] = \cpx[x_0,x_1,\ldots,x_n]$
are defined similarly.
A polynomial is called a form if it is homogeneous.
The $\re[x]_d$ (resp. $\re[\tilde{x}]_d$) denotes the subspace of homogeneous polynomials
in $\re[x]$ (resp. $\re[\tilde{x}]$) of degree $d$,
and $\re[x]_{\leq d} = \re[x]_0+\re[x]_1+\cdots+\re[x]_d$.
For a polynomial $f(x)$ of degree $d$, $f^h(\tilde{x})$ denotes its homogenization $x_0^df(x/x_0)$.
For a tuple $g=(g_1,\ldots,g_m)$ of polynomials, denote $g^h=(g_1^h,\ldots,g_m^h)$.
For a finite set $S$, $|S|$ denotes its cardinality.
For a general set $S \subseteq \re^n$,
$int(S)$ denotes its interior, and $\pt S$ denotes its boundary in standard Euclidean topology.
For a matrix $A$, $A^T$ denotes its transpose.
For a symmetric matrix $X$, $X\succeq 0$ (resp., $X\succ 0$) means
$X$ is positive semidefinite (resp. positive definite).
For $u\in \re^N$, $\| u \|_2 = \sqrt{u^Tu}$ denotes the standard Euclidean norm.
%

\subsection{Ideals and varieties}

In this subsection we give a brief review about ideals and varieties
in elementary algebraic geometry.
We refer to \cite{CLO97,Ha} for more details.

A subset $I$ of $\cpx[x]$ is called an {\it ideal}
if $p \cdot q \in I$ for any $p\in \re[x]$ and $q\in I$.
For $g_1,\ldots,g_m \in \cpx[x]$, $\langle g_1,\cdots,g_m \rangle $ denotes
the smallest ideal containing every $g_i$.
The $g_1,\ldots,g_m$ are called generators of
$\langle g_1,\cdots,g_m \rangle $, or equivalently,
$\langle g_1,\ldots,g_m \rangle $ is generated by $g_1,\ldots,g_m$.
Every ideal in $\cpx[x]$ is generated by a finite number of polynomials.

An {\it algebraic variety} is a subset of $\cpx^n$ that is defined by
a finite set of polynomial equations.
Sometimes, an algebraic variety is just called a variety.
Let $g=(g_1,\ldots, g_m)$ be a tuple of polynomials in $\re[x]$. Define
\[
V(g) = \{x\in \cpx^n:  g_1(x)=\cdots=g_m(x)=0\}.
\]
In optimization, we are more interested in real solutions. Define
\[
V_\re(g) = \{x\in \re^n:  g_1(x)=\cdots=g_m(x)=0\}.
\]
It is called a {\it real algebraic variety}. Clearly, $V_\re(g) \subset V(g)$.
If $I = \langle g_1, \ldots , g_m \rangle$, we define $V(I) = V(g)$.

Given $V \subseteq \cpx^n$, the set of all polynomials vanishing on $V$ is an ideal
and denoted by
\[
I(V)  =  \{ h\in \cpx[x]\,:\, h(u)=0 \,\,\,\forall \,\, u \in V\}.
\]
Clearly, if $ V=V(I)$ and $p \in I$, then $p \in I(V)$.
The following is a reverse to this fact.

\begin{theorem} [Hilbert's Nullstellensatz] \label{HilNull}
Let $I\subset \cpx[x]$ be an ideal. If $p \in I(V)$,
then $p^k\in I$ for some integer $k>0$.
\end{theorem}

Given a subset $S \subset \cpx^n$, the smallest variety $V\subset \cpx^n$ containing $S$
is called the {\it Zariski} closure of $S$, and is denoted by $Zar(S)$.
For instance, for $S=\{x\in \re^2: x_1^2+x_2^3=1, x_1\geq 0, x_2 \geq 0\}$,
its Zariski closure is the variety $\{x\in \cpx^2: x_1^2+x_2^3=1\}$.
In the Zariski topology on $\cpx^n$, the varieties are closed sets,
and the complements of varieties are open sets.

The varieties in the above are also called {\it affine} varieties,
because they are defined in the vector space $\cpx^n$ or $\re^n$.
We also need projective varieties that are often more convenient in algebraic geometry.
Let $\P^n$ be the $n$-dimensional complex projective space,
where each point $\tilde x \in \P^n$
is a family of nonzero vectors $ \tilde{x}=(x_0,x_1, \ldots,x_n)$ that are parallel to each other.
A set $U$ in $\P^n$ is called a {\it projective variety}
if it is defined by finitely many homogeneous polynomial equations.
For given forms $p_1(\tilde{x}), \ldots, p_m(\tilde{x})$,
denote the projective variety
\[
V_\P(p_1, \ldots, p_m)= \left\{ \tilde{x} \in \P^n:\,
p_1(\tilde{x}) = \cdots = p_r(\tilde{x})=0\right\}.
\]
In particular, if $m=1$, $V_\P(p_1)$ is called a hypersurface.
Furthermore, if $p_1$ has degree one, $V_\P(p_1)$ is called a hyperplane.
In the Zariski topology on $\P^n$, the projective varieties are closed sets,
and their complements are open sets.

A variety $V$ is irreducible if there exist no
proper subvarieties $V_1,V_2$ of $V$ such that $V= V_1 \cup V_2$.
The dimension of a variety $U$ is the biggest integer $\ell$ such that
$U=U_{0}\supset U_{1}\supset \cdots \supset U_{\ell}$
where every $U_i$ is an irreducible variety.
For an ideal $I \subseteq \cpx[x]$,
its dimension is defined to be the dimension of its variety $V(I)$.
It is zero-dimensional if and only if $V(I)$ is finite.

Let $V$ be a projective variety of dimension $\ell$ in $\P^n$ and
$I(V)= \langle f_1, \ldots, f_r \rangle$.
The {\it singular locus} $V_{sing}$ is defined to be the variety
\[
V_{sing}= \left\{ w \in V: \,  \text{ rank }
J(f_1, \ldots, f_r) \, <  \, n-\ell\, \, \text{at} \, w \right\},
\]
where $J(f_1, \ldots, f_r)$ denotes the Jacobian matrix of $f_1,\ldots,f_r$.
The points in $V_{sing}$ are called singular points of $V$.
If $V_{sing} = \emptyset$, we say $V$ is smooth.
When $V$ is an affine variety, its singular locus and singular points
are defined similarly.

\subsection{Discriminants and resultants}

In this subsection, we review some basics about discriminants and resultants
for multivariate polynomials. We refer to \cite{GKZ} for more details.

Let $f(x)$ be a polynomial in $x = (x_1,\ldots,x_n)$
and $u \in \cpx^n$ be a complex zero point of $f(x)$, i.e., $f(u)=0$.
We say $u$ is a critical zero of $f$ if $\nabla_{x}f(u)=0$.
Not every polynomial has a critical complex zero.
In the univariate case, if $f(x)=ax^2+bx+c$ is quadratic
and has a critical complex zero, then its discriminant $b^2-4ac = 0$.
In the multivariate case, if $f(x)=x^TAx$ is quadratic and $A$ is symmetric, then
$f(x)$ has a nonzero complex critical point if and only if its determinant $\det (A) = 0$.
The above can be generalized to polynomials of higher degrees.
In \cite{GKZ}, the discriminants have been defined for general multivariate polynomials.

For convenience, let $f(x)$ be a form in $x = (x_1,\ldots,x_n)$.
The discriminant $\Delta(f)$ is a polynomial in the coefficients of $f$ satisfying
\[
\Delta(f) = 0 \quad  \Longleftrightarrow \quad
\exists \, u\in \cpx^n\backslash \{0\}:
\nabla f(u)=0.
\]
The discriminant $\Delta(f)$ is homogeneous,
irreducible and has integer coefficients.
It is unique up to a sign if all its integer coefficients are coprime.
When $\deg(f)=d$, $\Delta(f)$ has degree $n(d-1)^{n-1}$.
For instance, when $n=2$ and $d=3$, we have the formula (see \cite[Chap.~12]{GKZ})
\[
\Delta(ax_1^3+bx_1^2x_2+cx_1x_2^2+dx_2^3) =
b^2c^2-4ac^3-4b^3d+18abcd-27a^2d^2.
\]

A more general definition than discriminant is {\it resultant}.
Let $f_1,\ldots, f_n$ be forms in $x \in \re^n$.
The resultant $Res(f_1,\ldots, f_n)$ is a polynomial
in the coefficients of $f_1, \ldots,f_n$ satisfying
\[
Res(f_1,\ldots,f_n) = 0 \quad  \Longleftrightarrow \quad
\exists \, u\in \cpx^n\backslash \{0\}:
f_1(u)=\cdots = f_n(u)=0.
\]
The polynomial $Res(f_1,\ldots, f_n)$ is homogeneous, irreducible and has integer coefficients.
It is unique up to a sign if all its coefficients are coprime.
If $f_i$ has degree $d_i$, then $Res(f_1,\ldots, f_n)$ is homogeneous in every $f_k$
of degree $d_1\cdots d_{k-1}d_{k+1}\cdots d_n$, and its total degree is
\[
d_1\cdots d_n \left( d_1^{-1} + \cdots + d_n^{-1}\right).
\]
In case of $n=2$, a general formula for $Res(f_1,\ldots, f_n)$ is given in \cite[Sec.~4.1]{Stu02}.
For instance, if $f_1(x) =ax_1^2+bx_1x_2+cx_2^2$ and $f_2(x) = dx_1^2+ex_1x_2+fx_2^2$, then
\[
Res(f_1,f_2) = c^2d^2-bcde+ace^2+b^2df-2acdf-abef+a^2f^2.
\]

We would like to to remark that the discriminant is a specialization of resultant.
A form $f(x)$ has a nonzero complex critical point if and only if
\[
\frac{\pt f(x)}{\pt x_1} = \cdots = \frac{\pt f(x)}{\pt x_n} =0
\]
has a nonzero complex solution.
So $\Delta(f) = \eta \cdot Res(\frac{\pt f}{\pt x_1}, \ldots, \frac{\pt f}{\pt x_n})$
for a scalar $\eta \ne 0$.

In many situations, we often handle nonhomogeneous polynomials.
The discriminants and resultants would also be defined for them.
Let $f(x)$ be a general polynomial in $x = (x_1,\ldots,x_n)$,
and the form $f^h(\tilde{x})$ in $\tilde{x}=(x_0,x_1,\ldots,x_n)$ be its homogenization.
The discriminant $\Delta(f)$ of $f(x)$ is then defined to be $\Delta(f^h)$ .
Observe that if $u\in \cpx^n$ is a critical zero point of $f$, i.e., $f(u)=0$ and $\nabla_x f(u)=0$,
then we must have $\nabla_{\tilde{x}} f^h(\tilde{u})=0$. Here $\tilde{u}=(1,u_1,\ldots,u_n)$.
To see this point, recall the Euler's formula (suppose $\deg(f)=d$)
\be \label{fm:euler}
d \cdot f^h(\tilde{x}) =  x_0 \frac{\pt f^h(\tilde{x})}{\pt x_0} +
x_1 \frac{\pt f^h(\tilde{x})}{\pt x_1} + \cdots + x_n \frac{\pt f^h(\tilde{x})}{\pt x_n}.
\ee
Since $f^h(\tilde{u}) = f(u),
\frac{\pt f^h(\tilde{u})}{\pt x_1} = \frac{\pt f(u)}{\pt x_1}, \ldots,
\frac{\pt f^h(\tilde{u})}{\pt x_n} = \frac{\pt f(u)}{\pt x_n},
$ it holds that $\nabla_{\tilde{x}} f^h(\tilde{u})=0$.
It is possible that $\Delta(f)=0$ while $f$ does not have a critical zero point,
because $\nabla_{\tilde{x}} f^h(\tilde{x})=0$ might have a solution at infinity $x_0=0$.

The resultants are similarly defined for nonhomogeneous polynomials.
Let $f_0,f_1,\ldots, f_n$ be general polynomials in $x = (x_1,\ldots,x_n)$.
The resultant $Res(f_0,f_1,\ldots, f_n)$ is then defined to be
$Res(f_0^h,f_1^h,\ldots,f_n^h)$.
Here each form $f_i^h(\tilde{x})$ is the homogenization of $f_i(x)$.
Clearly, if the polynomial system
\[
f_0(x)=f_1(x)=\cdots=f_n(x)=0
\]
has a solution in $\cpx^n$, then the homogeneous system
\[
f_0^h(\tilde{x})=f_1^h(\tilde{x})=\cdots=f_n^h(\tilde{x})=0
\]
has a solution in $\P^n$.
The reverse is not always true, because the latter might have a solution at infinity $x_0=0$.

There are systemic procedures to compute resultants (hence including discriminants)
for general polynomials.
We refer to \cite[Chap.~3]{CLO98}, \cite[Sec.~4, Chap.~3]{GKZ}, and \cite[Chap.~4]{Stu02}.

\section{Discriminants of several polynomials} \label{sec:dis-sevral}
\setcounter{equation}{0}

In this section, we assume $f_0(\tilde{x}),f_1(\tilde{x}),\ldots, f_m(\tilde{x})$
are forms in $\tilde{x}=(x_0,x_1,\ldots,x_n)$ of degrees $d_0,d_1,\ldots,d_m$ respectively, and $m \leq n$.
Denote $f=(f_0,f_1,\ldots,f_m)$.
If every $f_i$ has generic coefficients, the polynomial system
\be \label{plsys:f0m}
f_0(\tilde{x}) = \cdots = f_m(\tilde{x}) = 0
\ee
has no singular solution in $\P^n$,
that is, for any $\tilde{u} \in \P^n$ satisfying \reff{plsys:f0m}, the Jacobian
\[
J_f (\tilde{u}):=\bbm \nabla_{\tilde{x}} f_0(\tilde{u}) & \nabla_{\tilde{x}} f_1(\tilde{u}) &  \cdots
& \nabla_{\tilde{x}} f_m(\tilde{u}) \ebm
\]
has full rank. For some particular $f$, \reff{plsys:f0m} might have a singular solution. Define
\[
W(d_0,\ldots,d_m) =
\left\{(f_0,\ldots,f_m) \in \prod_{i=0}^m \re[\tilde{x}]_{d_i}:
\baray{c}
\exists \tilde{u} \in \P^n \quad s.t. \\
f_0(\tilde{u})=\cdots=f_m(\tilde{u})=0 \\
 \rank J_f(\tilde{u}) \leq m
\earay
\right\}.
\]
When every $d_i=1$, $W(1,\ldots,1)$ consists of all vector tuples $(f_0,\ldots,f_m)$
such that $f_0,\ldots,f_m$ are linearly dependent. Thus
$W(1,\ldots,1)$ consists of all $(n+1)\times (m+1)$ matrices whose ranks are at most $m$,
which is a determinantal variety of codimension $n+1-m$.
It is not a hypersurface when $m \leq n-1$.
When every $d_i=d>1$, $W(d,\ldots,d)$ consists of all tuples
$(f_0,\ldots,f_m)$ such that the multi-homogeneous form in $(\tilde{x}, \tilde{\lmd})$
(here $\tilde{\lmd}=(\lmd_0,\lmd_1,\ldots,\lmd_m)$)
\[
\mc{L}(\tilde{x},\tilde{\lmd}):=
\lmd_0 f_0(\tilde{x}) + \lmd_1 f_1(\tilde{x}) + \cdots + \lmd_m f_m(\tilde{x})
\]
has a critical point in the product of projective spaces $\P^n \times \P^m$.
As is known, the multi-homogeneous form $\mc{L}(\tilde{x},\tilde{\lmd})$
has a critical point in $\P^n \times \P^m$ if and only if its discriminant
vanishes (see \cite[Section~2B, Chap. 13]{GKZ}).
So $W(d,\ldots,d)$ is a hypersurface.
When the $d_i$'s are not equal and at least one $d_i>1$,
$W(d_0,\ldots,d_m)$ is also a hypersurface,
which is a consequence of Theorem~4.8 of Looijenga \cite{Loo}.
This fact was kindly pointed out to the author by Kristian Ranestad.
So we assume at least one $d_i >1$, and then $W(d_0,\ldots,d_m)$ is a hypersurface.
Let $\Delta(f_0,f_1,\ldots,f_m)$ be a defining polynomial of the lowest degree for $W(d_0,\ldots,d_m)$.
It is unique up to a constant factor and satisfies
\be \label{dis:f0m}
(f_0,\ldots,f_m) \in W(d_0,\ldots,d_m) \quad \Longleftrightarrow \quad
\Delta(f_0,f_1,\ldots,f_m) = 0.
\ee
For convenience, we also call $\Delta(f_0,f_1,\ldots,f_m)$ the discriminant of
forms $f_0(\tilde{x}),\ldots, f_m(\tilde{x})$.
When $m=0$, $\Delta(f_0,f_1,\ldots,f_m)$ becomes the standard discriminant of a single form,
which has degree $(n+1)(d_0-1)^n$. So $\Delta(f_0,f_1,\ldots,f_m)$
can be thought of as a generalization of $\Delta(f_0)$.
In the rest of this section, we are going to prove
a general degree formula for $\Delta(f_0,f_1,\ldots,f_m)$.

For every integer $k\geq 0$, denote by $S_k$ the $k$-th complete symmetric polynomial
\[
S_k(a_1,\ldots,a_t) = \sum_{i_1+\cdots+i_t=k}
a_1^{i_1} \cdots a_t^{i_t}.
\]
Let $H(\tilde{x}) \in \re[\tilde{x}]^{(n+1) \times (m+1)}$ be a matrix polynomial
such that its every entry $H_{ij}(\tilde{x})$ is homogeneous
and all the entries of its every column have the same degree. Define
\be  \label{DetVar:H(x)}
\mc{D}_m(H) = \{ \tilde{x} \in \P^n: \mbox{rank} \, H(\tilde{x}) \leq m \}.
\ee

\begin{theorem}  \label{thm-deg:dis-itsc}
Suppose every $d_i>0$, at least one $d_i>1$, and $m \leq n$.
Then the discriminant $\Delta(f_0,\ldots,f_m)$ has the following properties:
\bit

\item [a)] For every $k=0,\ldots,m$, $\Delta(f_0,f_1,\ldots,f_m)$
is homogeneous in $f_k$. It also holds that
\[
\Delta(f_0,\ldots,f_m) = 0 \quad \mbox{ whenever } f_i = f_j \, \mbox{ for } i \ne j.
\]

\item [b)] For every $k=0,\ldots,m$,
the degree of $\Delta(f_0,f_1,\ldots,f_m)$ in $f_k$ is
\be \label{deg:DisSigVar}
d_0\cdots \check{d_k} \cdots d_m  \cdot
S_{n-m}\Big(d_0-1, \ldots, \widehat{\widehat{d_k-1}}, \ldots, d_m-1\Big).
\ee
In the above, $\check{d_k}$ means $d_k$ is missing,
and $\widehat{\widehat{a}}$ means $a$ is repeated twice.
Thus the total degree of $\Delta(f_0,f_1,\ldots,f_m)$ is
\be \label{tdeg:DisSigVar}
d_0 \cdots d_m
\left( \sum_{k=0}^m \frac{1}{d_k}
S_{n-m}\Big(d_0-1, \ldots, \widehat{\widehat{d_k-1}}, \ldots, d_m-1\Big) \right).
\ee

\item [c)] For fixed $f_1,\ldots, f_m$,  $\Delta(f_0,f_1,\ldots,f_m)$
is identically zero in $f_0$ if and only if the projective variety $V_\P(f_1,\ldots,f_m)$
has a positive dimensional singular locus.

\eit
\end{theorem}
\begin{proof}
a) Note that for any scalar $\af \ne 0$, $(f_0, \ldots, f_m) \in W(d_0,\ldots,d_m)$ if and only if
\[
(f_0, \ldots, f_{k-1}, \af f_k, f_{k+1}, \ldots, f_m) \in W(d_0,\ldots,d_m).
\]
So, by relation \reff{dis:f0m}, $\Delta(f_0,\ldots,f_m)$ must be homogeneous in every $f_k$.

If $f_i = f_j$ for some distinct $i,j$, say $i=0,j=1$, then
$(f_0, \ldots, f_m) \in W(d_0,\ldots,d_m)$
because the polynomial system \reff{plsys:f0m} must have a solution in $\P^n$
(it has only $m-1<n$ distinct equations) and its Jacobian is singular
(its first two columns are same).

b) For convenience, we only prove the degree formula for $k=0$.
Choose generic forms $f_0,\ldots, f_m$ of degrees $d_0,\ldots,d_m$ respectively,
and another generic form $h$ of degree $d_0$.
Then the degree of $\Delta(f_0,f_1,\ldots,f_m)$ in $f_0$
is equal to the number of scalars $\gamma$ such that
\be  \label{dis:f+gma*h}
\Delta(f_0+\gamma h,f_1,\ldots,f_m) = 0.
\ee
Since the $f_i$'s are generic, $\Delta(f_1,\ldots,f_m) \ne 0$ and hence
$V_\P(f_1,\ldots,f_m)$ is nonsingular.
\begin{lem} \label{deg:gama}
The condition \reff{dis:f+gma*h} is equivalent to
\be  \label{sigJac:f+h}
\baray{c}
\exists \, u \in \P^n,\, \exists\, \gamma \in \cpx: \quad  f_1(u)=\cdots =f_m(u)=0, \\
\mbox{rank} \bbm \nabla_{\tilde{x}} f_0(u) +\gamma \nabla_{\tilde{x}} h(u)
& \nabla_{\tilde{x}} f_1(u) & \cdots & \nabla_{\tilde{x}} f_m(u) \ebm \leq m.
\earay
\ee
Furthermore, every $u$ satisfying \reff{sigJac:f+h} determines $\gamma$ uniquely.
\end{lem}
\noindent {\it Proof.}
By relation \reff{dis:f0m},  \reff{dis:f+gma*h} clearly implies \reff{sigJac:f+h}.
So we only prove the reverse.
Suppose \reff{sigJac:f+h} is satisfied by some $u$ and $\gamma$.
The rank condition in \reff{sigJac:f+h} implies there exists
$(\mu_0, \mu_1, \ldots, \mu_m) \ne 0$ satisfying
\[
\mu_0 \big(\nabla_{\tilde{x}} f_0(u) +\gamma \nabla_{\tilde{x}} h(u)\big)+
\mu_1 \nabla_{\tilde{x}} f_1(u) +\cdots + \mu_m \nabla_{\tilde{x}} f_m(u) = 0.
\]
Since $V_\P(f_1,\ldots,f_m)$ is nonsingular, we must have $\mu_0 \ne 0$ and can scale $\mu_0=1$.
By Euler's formula \reff{fm:euler}, premultiplying $u^T$ in the above equation gives
\[
d_0(f_0(u) +\gamma h(u)) + \mu_1 d_1f_1(u) +\cdots + \mu_m  d_m f_m(u) = 0.
\]
Thus the equations in \reff{sigJac:f+h} imply
$
f(u) + \gamma h(u) = 0,
$
and thus \reff{dis:f+gma*h} holds by relation \reff{dis:f0m}.

Now we prove each $u$ in \reff{sigJac:f+h} uniquely determines $\gamma$.
If $h(u) \ne 0$, we know $\gamma = - f(u)/h(u)$ from the above.
If $h(u) = 0$, because $V_\P(h,f_1,\ldots,f_m)$ is nonsingular ($h$ and $f_i$ are all generic),
we can generally assume the first $m+1$ rows of the Jacobian of
$h, f_1, \ldots, f_m$ at $u$ are linearly independent,
which is denoted by $\bbm b & F\ebm$ with $b\in \cpx^{m+1}$ and $F\in \cpx^{(m+1)\times m}$.
Denote by $a$ the first $m+1$ entries of $ \nabla_{\tilde{x}} f_0(u)$.
Then, $\det \bbm  b & F \ebm \ne 0$ and \reff{sigJac:f+h} implies
\[
\det \bbm  a+\gamma b & F \ebm =
\det \bbm  a  & F \ebm  + \gamma \det \bbm  b & F \ebm =0.
\]
So $\gamma = - \det \bbm  a  & F \ebm / \det \bbm  b & F \ebm$.
There is a unique $\gamma$ for every $u$ in \reff{sigJac:f+h}.
\hfill
\qed

Clearly, \reff{sigJac:f+h} is also equivalent to
\[
\baray{c}
\exists u \in \P^n: \quad  f_1(u)=\cdots =f_m(u)=0, \\
\rank \bbm  \nabla_{\tilde{x}} f_0(u) & \nabla_{\tilde{x}} h(u)
& \nabla_{\tilde{x}} f_1(u) & \cdots & \nabla_{\tilde{x}} f_m(u) \ebm \leq m+1.
\earay
\]
Let $J$ be the Jacobian matrix in the above. By Lemma~\ref{deg:gama},
the degree of $\Delta(f_0,f_1,\ldots,f_m)$ in $f_0$ is equal to the cardinality of
\[
U:=\mc{D}_{m+1}(J) \cap V_\P(f_1,\ldots,f_m).
\]
The variety $V_\P(f_1,\ldots,f_m)$ is smooth, has codimension $m$ and degree $d_1\cdots d_m$.
Since every $f_i$ and $h$ are generic, $\mc{D}_{m+1}(J)$
is also smooth, has dimension $m$ and intersects $V_\P(f_1,\ldots,f_m)$ transversely.
So $U$ is a finite variety.
We refer to Proposition~2.1 and Theorem~2.2 in \cite{NR09} for more details about this fact.
The degree of the determinantal variety $\mc{D}_{m+1}(J)$ is (cf. Proposition A.6 of \cite{NR09})
\[
S_{n-m}(d_0-1, d_0-1, d_1-1, \ldots, d_m-1).
\]
By B\'{e}zout's theorem (cf. Proposition A.3 of \cite{NR09}, or \cite{Ha}),
the degree of $U$ is given by the formula \reff{deg:DisSigVar},
which also equals its cardinality.
Therefore, the degree of $\Delta(f_0,f_1,\ldots,f_m)$ in $f_0$
is given by \reff{deg:DisSigVar},
and then the formula for its total degree immediately follows.

c) Clearly, if the singular locus $V_\P(f_1,\ldots,f_m)_{sing}$ has positive dimension,
then it must intersect the hypersurface
$f_0(\tilde{x}) =0$ for arbitrary $f_0$, by B\'{e}zout's theorem.
Thus the system \reff{plsys:f0m} has a singular solution,
which implies $\Delta(f_0,f_1,\ldots,f_m)=0$ for arbitrary $f_0$.
To prove the reverse, suppose $\Delta(f_0,f_1,\ldots,f_m)=0$ is identically zero in $f_0$.
We need to show that $V_\P(f_1,\ldots,f_m)_{sing}$ has positive dimension.
For a contradiction, suppose it is zero dimensional
and consists of finitely many points $u^{(1)},\ldots, u^{(N)} \in \P^n$.
Note that the dimension of the set
\[
T = \bigcup_{v\in V_\P(f_1,\ldots,f_m)} \mathcal{R} \big(
\bbm \nabla_{\tilde{x}} f_1(v) & \cdots & \nabla_{\tilde{x}} f_m(v) \ebm \big)
\]
is at most $n$ in the affine space $\re^{n+1}$ whose dimension is $n+1$.
Here $\mathcal{R} \big(A\big)$ denotes the column range space of matrix $A$.
So the complement $\re^{n+1}\backslash T$ has positive dimension,
and hence we can choose $a \in \re^{n+1}\backslash T$
such that the hyperplane $a^T\tilde{x} = 0$ does not pass through $u^{(1)},\ldots, u^{(N)}$.
For $f_a(\tilde{x}) = a^T\tilde{x}$, the homogeneous polynomial system
\[
f_a(\tilde{x})=f_1(\tilde{x})=\cdots = f_m(\tilde{x})=0
\]
has no singular solution in $\P^n$,
which implies $\Delta(f_a,f_1,\ldots,f_m) \ne 0$ by \reff{dis:f0m} and then contradicts that
$\Delta(f_0,f_1,\ldots,f_m)=0$ is identically zero in $f_0$.
So, the singular locus of $V_\P(f_1,\ldots,f_m)$ must have positive dimension.
\end{proof}

The discriminant $\Delta(f_0,\ldots,f_m)$ of $m+1$ forms
$f_0(\tilde{x}), \ldots, f_m(\tilde{x})$
is a natural generalization of the standard discriminant of a single form.
In formula \reff{tdeg:DisSigVar}, if we set $m=0$, then the degree of $\Delta(f_0)$ is
$(n+1) (d_0-1)^n$, which is precisely the degree of discriminants of forms of degree $d_0$ in $n+1$ variables.

In Theorem~\ref{thm-deg:dis-itsc}, if every $d_i=d$,
the discriminant $\Delta(f_0,\ldots,f_m)$ is homogeneous in every $f_i$ of degree
$
\binom{n+1}{m+1} d^m (d-1)^{n-m},
$
and its total degree is
$
(n+1)\binom{n}{m} d^m (d-1)^{n-m}.
$
This is precisely the degree of the discriminant
of the multi-homogeneous form $\mc{L}(\tilde{x},\tilde{\lmd})$
(see Theorem 2.4 of Section 2B in Chapter 13 of \cite{GKZ}).

In Theorem~\ref{thm-deg:dis-itsc}, when $m=n$,
the Jacobian of \reff{plsys:f0m} must be singular at its every solution $\tilde{u}\in \P^n$,
because by Euler's formula \reff{fm:euler}
\[
\tilde{u}^T \bbm \nabla_{\tilde{x}} f_0(\tilde{u}) & \cdots & \nabla_{\tilde{x}} f_n(\tilde{u}) \ebm
= \bbm d_0 f_0(\tilde{u}) & \cdots & d_n f_n(\tilde{u}) \ebm =0.
\]
So \reff{plsys:f0m} has a singular solution if and only if the homogeneous polynomial system
\[
f_0(\tilde{x}) = \cdots = f_n(\tilde{x}) = 0
\]
has a solution in $\P^n$,
which is equivalent to that the resultant $Res(f_0, \ldots, f_n)$ vanishes. So
\[
\Delta(f_0,\ldots,f_n) \, = \, 0 \quad \Longleftrightarrow \quad
Res(f_0, \ldots, f_n) \,= \, 0.
\]
Observe that $\Delta(f_0,\ldots,f_n)$ and $Res(f_0, \ldots, f_n)$ have the same degree
\[
d_0 \cdots d_n
\left( d_0^{-1} + \cdots + d_n^{-1} \right).
\]
So $\Delta(f_0,\ldots,f_n)$ is equal to $Res(f_0, \ldots, f_n)$ up to a constant factor.

When $d_0>1$ and every $f_i(\tilde{x})=f_i^T\tilde{x} \,(1\leq i\leq m)$ is linear,
\reff{plsys:f0m} has a singular solution if and only if
$f_0(\tilde{x})$ has a nonzero critical point in the orthogonal complement
of the subspace $\mbox{span}\{f_1,\ldots,f_m\}$.
If every $f_i(\tilde{x})=x_{i-1}$, $\Delta(f_0, x_0,\ldots,x_{m-1})$
vanishes if and only if $\Delta(\hat{f})=0$.
Here $\hat{f} = f(0,\ldots,0,x_m,\ldots,x_n)$ is a form in $(x_m,\ldots,x_n)$.
Since $\Delta(f_0, x_0,\ldots,x_{m-1})$ has degree
$(n-m+1)(d_0-1)^{n-m}$ in $f_0$, we have
\be \label{copoly:dis}
\Delta(f_0, x_0,\ldots,x_{m-1}) = \eta \cdot \Delta(\hat{f})
\ee
for some scalar $\eta \ne 0$.
Furthermore, if $f_0=\tilde{x}^TA\tilde{x}$ is quadratic, then it holds that
\be \label{quad:dis=det}
\Delta(\tilde{x}^TA\tilde{x}, x_0,\ldots,x_{m-1})
= \eta \cdot \det A(m+1:n+1, m+1:n+1).
\ee
Here $A(I,I)$ denotes the submatrix of $A$
whose row and column indices are from $I$.

We conclude this section by generalizing $\Delta(f_0,\ldots,f_m)$ to
nonhomogeneous polynomials. If $f_0,\ldots,f_m$ are not forms,
denote by $f_i^h$ the homogenization of $f_i$.
Then $\Delta(f_0,\ldots,f_m)$ is defined to be $\Delta(f_0^h,\ldots,f_m^h)$.

\section{Polynomials nonnegative on $\re^n$} \label{sec:psd}
\setcounter{equation}{0}

This section studies the cone $P_d(K)$ when $K=\re^n$.
Note that a polynomial $f(x)$ is nonnegative in $\re^n$ if and only if
its homogenization $f^h(\tilde{x})$ is nonnegative everywhere.
So we just consider the cone of nonnegative forms.

Let $P_{n,d}$ be the cone of forms nonnegative in $\re^n$ of degree $d$.
Here $d>0$ is even.
Clearly, a form $f$ lies in the interior of $P_{n,d}$ if and only if it is positive definite,
that is, $f(x) > 0$ for every $x\ne 0$.
If $f(x)$ lies on the boundary $\pt P_{n,d}$, then it vanishes at some $0\ne u \in\re^n$.
Since $f(x)$ is nonnegative everywhere, $u$ must be a minimizer of $f(x)$ and $\nabla f(u)=0$.
This implies that $f(x)$ has a nonzero critical point, and hence its discriminant $\Delta(f) = 0.$
So the boundary $\pt P_{n,d}$ lies on the discriminantal hypersurface
\[
\mc{E}_{n,d} = \{f \in \re[x]_d: \Delta(f) = 0 \}.
\]

\begin{theorem}
The Zariski closure of the boundary $\pt P_{n,d}$ is $\mc{E}_{n,d}$,
which is an irreducible hpyersurface of degree $n(d-1)^{n-1}$.
\end{theorem}
\begin{proof}
The discriminant $\Delta(f)$ is irreducible and has degree $n(d-1)^{n-1}$,
so the hypersurface $\mc{E}_{n,d}$ is also irreducible
and has degree $n(d-1)^{n-1}$. Since $\pt P_{n,d} \subset \mc{E}_{n,d}$,
its Zariski closure $Zar(\pt P_{n,d})$ lies on $\mc{E}_{n,d}$.
The irreducibility of $\mc{E}_{n,d}$ implies $Zar(\pt P_{n,d})=\mc{E}_{n,d}$.
\end{proof}

When $d=2$, $P_{n,2}$ reduces to the cone of positive semidefinite matrices.
A typical barrier for $P_{n,2}$ is $-\log \det A$,  where $f(x) = x^TAx$.
%
Does there exist a similar barrier for $P_{n,d}$ when $d>2$?
Unfortunately, this is impossible if we require the barrier to be of log-polynomial type,
as will be shown in the below.

Let $\lmd_{min}(f)$ denote the smallest value of a form $f(x)$ on the unit sphere
\be \label{lmd-min:R^n}
\lmd_{min}(f) := \min_{\|x\|_2=1}  f(x).
\ee
The boundary $\pt P_{n,d}$ is then characterized by
$
\lmd_{min}(f) = 0.
$
Clearly, if $\lmd_{min}(f)=0$ then $\Delta(f)=0$, but the reverse might not be true.
For instance, for the positive definite form $\hat{f}(x) = \|x\|_2^d$ (for even $d>2$),
$\lmd_{min}(\hat{f})=1$ but $\Delta(\hat{f})=0$,
because $\nabla \hat{f}(x)=0$ has a nonzero complex solution.
So the discriminantal hypersurface $\Delta(f) = 0$
intersects the interior of $P_{n,d}$ when $d>2$ is even.
This interesting fact leads to the following theorem.

\begin{theorem}  \label{psd:nobarfun}
If $d>2$ is even and $n\geq 2$, there is no polynomial $\varphi(f)$ satisfying
\bit

\item $\varphi(f)>0$ whenever $f$ lies in the interior of $P_{n,d}$, and

\item $\varphi(f)=0$ whenever $f$ lies on the boundary of $P_{n,d}$.

\eit
Therefore, $-\log \varphi(f)$ can not be a barrier function
for the cone $P_{n,d}$ when we require $\varphi(f)$ to be a polynomial,
and $P_{n,d}$ is not representable by a linear matrix inequality (LMI), that is, 
there is no symmetric matrix pencil 
\[
L(f) \, = \,  \sum_{\af\in\N^n: |\af|=d} \, f_\af A_\af  
\quad ( \mbox{ where } \, f(x) = \sum_\af \, f_\af x^\af ) 
\]
such that $ P_{n,d} = \left\{ f \in \re[x]_d: \, L(f) \succeq 0\right\}$
and $L(f) \succ 0$ for $f\in int(P_{n,d})$.  
\end{theorem}
\begin{proof}
For the first part, we prove by contradiction. Suppose such a $\varphi$ exists.
The zero set $\lmd_{min}(f)=0$ lies on the variety $V(\varphi)$.
Since the discriminantal hypersurface $\Delta(f)=0$ is the Zariski closure of $\lmd_{min}(f)=0$,
i.e., the smallest variety containing $\lmd_{min}(f)=0$,
$\Delta(f)=0$ is a subvariety of $V(\varphi)$.
So $\varphi(f)$ is vanishing on $\Delta(f)=0$.
By Hilbert Nullstenllensatz (see Theore~\ref{HilNull}),
there exist an integer $k>0$ and a polynomial $p(f)$ satisfying
\[
\varphi(f)^k = \Delta(f) \cdot p(f).
\]
Now we choose $\hat{f}(x) = \|x\|_2^{d} \in int(P_{n,d})$ in the above,
then $\Delta(\hat{f})=0$ and $\varphi(\hat{f})=0$, which contradicts the first item.

For the second part, the non-existence of $-\log$-polynomial type barrier function
immediately follows the first part of the theorem.
The non-existence of LMI representation also clearly follows the first part, 
because otherwise the determinant $\det L(f)$ 
would be a polynomial satisfying the first part.
\end{proof}

Theorem~\ref{psd:nobarfun} tells us that there does not exist a polynomial
$\varphi (f)$ such that $-\log \varphi (f)$ is a barrier for $P_{n,d}$.
However, $-\log \varphi (f)$ would be a barrier if $\varphi(f)$ is not required to be a polynomial.
Actually
\be  \label{logbar:lmd-min}
\phi(f) = -\log \lmd_{min}(f)
\ee
is a barrier for $P_{n,d}$,
where $\lmd_{min}(f)$ is defined by \reff{lmd-min:R^n}.
The function $\lmd_{min}(f)$ is semialgebraic, positive in $int(P_{n,d})$, and zero on $\pt P_{n,d}$.
The barrier $\phi(f)$ is also convex in $int(P_{n,d})$.
%
%
%
%

\begin{theorem}  \label{psd-bar:salg}
The function $\phi(f)$ is convex in $int(P_{n,d})$.
%
\end{theorem}
\begin{proof}
For any $f^{(1)}, f^{(2)} \in int(P_{n,d})$, from \reff{lmd-min:R^n} we have
\[
\lmd_{min}\left( \theta f^{(1)} + (1-\theta) f^{(2)} \right) \geq
\theta \lmd_{min}\left( f^{(1)}\right) + (1-\theta) \lmd_{min}\left( f^{(2)} \right),
\quad \forall \, \theta \in [0,1].
\]
Since $-\log(\cdot)$ is concave, the above then implies
\[
\phi\left( \theta f^{(1)} + (1-\theta) f^{(2)} \right) \leq
\theta \phi\left(f^{(1)}\right) + (1-\theta) \phi\left( f^{(2)}\right).
\]
So $\phi(f)$ is convex in $int(P_{n,d})$.
\end{proof}

However, the barrier $-\log \lmd_{min}(f)$ is not very useful
in practice, because computing $\lmd_{min}(f)$ is quite difficult.
When $d=4$, it is NP-hard to compute $\lmd_{min}(f)$.

\subsection{Computing the discriminantal variety $\Delta(f)=0$}

We have seen that $\pt P_{n,d}$ lies on the discriminantal hypersurface $\Delta(f)=0$.
Cayley's method would be applied to compute $\Delta(f)$, as introduced in Chap.~2 of \cite{GKZ}.
When $n=2$ and $d=4$, the boundary of $P_{2,4}$ lies on the hypersurface
defined by the polynomial
\[
\baray{c}
b^2c^2d^2-4ac^3d^2-4b^3d^3+18abcd^3-27a^2d^4-4b^2c^3e+16ac^4e+18b^3cde-80abc^2de \\
-6ab^2d^2e+144a^2cd^2e-27b^4e^2+144ab^2ce^2-128a^2c^2e^2-192a^2bde^2+256a^3e^3,
\earay
\]
where $a,b,c,d,e$ are the coefficients of
$f(x) = ax_1^4+bx_1^3x_2+cx_1^2x_2^2+dx_1x_2^3+ex_2^4$.
It is a homogenous polynomial of degree $6$ in $5$ variables.
When $n=3$ and $d=3$, $\Delta(f)$ is a homogeneous polynomial of degree $12$ in $20$ variables,
and has 21,894 terms in its full expansion.
When $n=3$ and $d=4$, $\Delta(f)$ is a form of degree $27$ in $15$ variables
and has thousands of terms.
A very nice method for computing discriminants of trivariate quartic forms
is described in Section~6 of \cite{SSS}.

Generally, it is quite complicated to compute $\Delta(f)$ directly.
A more practical approach for finding the discriminantal locus $\Delta(f)=0$
is to apply elimination theory (see \cite{CLO97}).
Let $f_p(x)$ be a form in $x$ whose coefficients are polynomial in a parameter $p = (a,b,...)$
over the rational field $\Q$, i.e., in the ring $\Q[p]$.
First, we dehomogenize $f_p(x)$ like
\[
g(1,x_2,\ldots,x_n) = f_p(1,x_2,\ldots,x_n).
\]
If $f_p(x) \in \pt P_{n,d}$ has no nontrivial critical point
on the hyperplane $x_1 = 0$ at infinity,
then the overdetermined polynomial system
\be \label{eq:g=grad=0}
g = \frac{\pt g}{\pt x_2}  = \cdots =  \frac{\pt g}{\pt x_n} = 0
\ee
must have a solution. Hence, we can use the elimination method described in \cite{CLO97}
to find the polynomial equation that the parameter $p$ satisfies.
By eliminating $x_2,\ldots,x_n$ in \reff{eq:g=grad=0},
we can get a polynomial $\varphi$ such that
if \reff{eq:g=grad=0} has a solution then $\varphi(p)=0$.
Hence, the discriminantal locus $\Delta(f_p) = 0$ lies on $\varphi(p) = 0$.
The polynomial $\varphi(p) = 0$ can be found by using function 
{\it elim} in software {\it Singular} \cite{Sing}.
%

\begin{exm}  \label{4exmp:psd-form}
(i) Consider the polynomials parameterized as
\[
f_{a,b}(x) = x_1^4 + x_2^4 + x_3^4
-a(x_1x_2^3 + x_2x_3^3 + x_3x_1^3) - b(x_1^3x_2+x_2^3x_3+x_3^3x_1).
\]
Its discriminant $\varphi(a,b) = \Delta(f_{a,b})$ is
\[
\baray{c}
16384(a+b-1)\cdot (a+b+2)^3\cdot (7a^2+7b^2-13ab+4a+4b+16)^4  \cdot \\
(7a^5+8ba^4-17a^4-14ba^3+16a^3b^2+16a^3-16a^2+48ba^2-21a^2b^2+16a^2b^3\\
+48ab^2-32ab-14ab^3+8ab^4-64a+7b^5-17b^4-16b^2+16b^3-64b+128)^3.
\earay
\]
The above formula is obtained by using a Maple code
that was kindly sent to the author by Bernd Sturmfels
for computing $(3,3,3)$-resultants. Let
\[
F = \left\{(a,b) \in \re^2: f_{a,b} \mbox{ is SOS in } x \right\}.
\]
It is a convex region in $\re^2$. The shape of $F$ would be found
by running the following Matlab code supported by software {\it YALMIP} \cite{YALMIP}
\begin{verbatim}
sdpvar x_1 x_2 x_3 a b;
p = x_1^4+x_2^4+x_3^4-a*(x_1*x_2^3+x_2*x_3^3+x_3*x_1^3)...
-b*(x_1^3*x_2+x_2^3*x_3+x_3^3*x_1);
v = monolist([x_1 x_2 x_3],2);
M = sdpvar(length(v));
L = [coefficients(p-v'*M*v,[x_1 x_2 x_3])==0,M>=0];
w = plot(L,[a,b],[1,1,1], 100);
fill(w(1,:),w(2,:),'b');
\end{verbatim}
The set $F$ is drawn in the shaded area of
the upper left picture in Figure~\ref{fig:psd-form}.
The curves there are defined by $\varphi(a,b) = 0$.
Since every nonnegative trivariate quartic form is SOS (see Reznick \cite{Rez00}),
we know $F=\left\{(a,b):f_{a,b}(x) \in P_{3,4}\right\}$.

\smallskip
\noindent
(ii) Consider the polynomials parameterized as
\[
\baray{c}
f_{a,b}(x) = x_1^4+x_2^4+x_3^4+x_4^4+
a(x_1^2x_2^2+x_2^2x_3^2-x_4^2x_1^2-x_3^2x_4^2) \\
+b(x_1^2x_3^2-x_2^2x_4^2+x_1x_2x_3x_4).
\earay
\]
Eliminating $x_2,x_3,x_4$ in \reff{eq:g=grad=0} gives $\varphi(a,b)$ as
\[
\baray{c}
(a+2)\cdot (a-2)\cdot (b+2)\cdot (b-2)\cdot (16a^2+16ab+5b^2+32a+16b+16) \cdot \\
(16a^2-16ab+5b^2-32a+16b+16) \cdot (4a^2b-8a^2-5b^2+16)(5b^2-16b+16).
\earay
\]
The curve $\Delta(f_{a,b})=0$ lies on $\varphi(a,b)=0$. Let
\[
F = \left\{(a,b) \in \re^2: f_{a,b} \mbox{ is SOS in } x \right\}.
\]
It is a convex region.
Using the method in (i), we get $F$ is the shaded area of
the upper right picture in Figure~\ref{fig:psd-form}.
The curves there are defined by $\varphi(a,b)=0$.
Let $G=\{(a,b):f_{a,b} \in P_{4,4}\}$.
Clearly, $F \subset G$ and the boundary of $G$ lies on $\varphi(a,b)=0$.
From the picture, we can see that $F$ is a maximal convex region
whose boundary lies on $\varphi(a,b)=0$. So $F=G$.

\begin{figure}[htb]
\centering
\btab{cc}
\includegraphics[height=.3\textwidth]{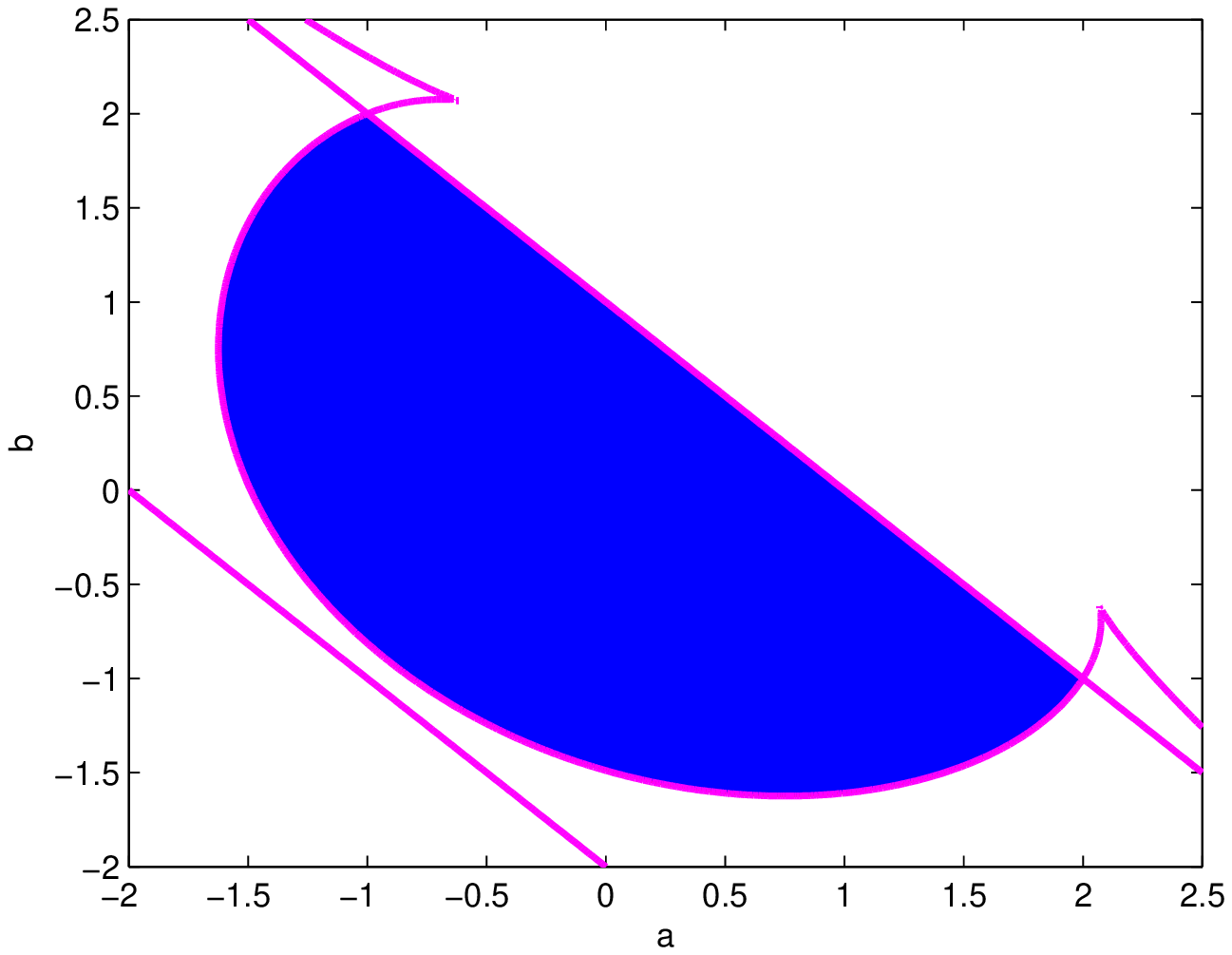} &
\includegraphics[height=.3\textwidth]{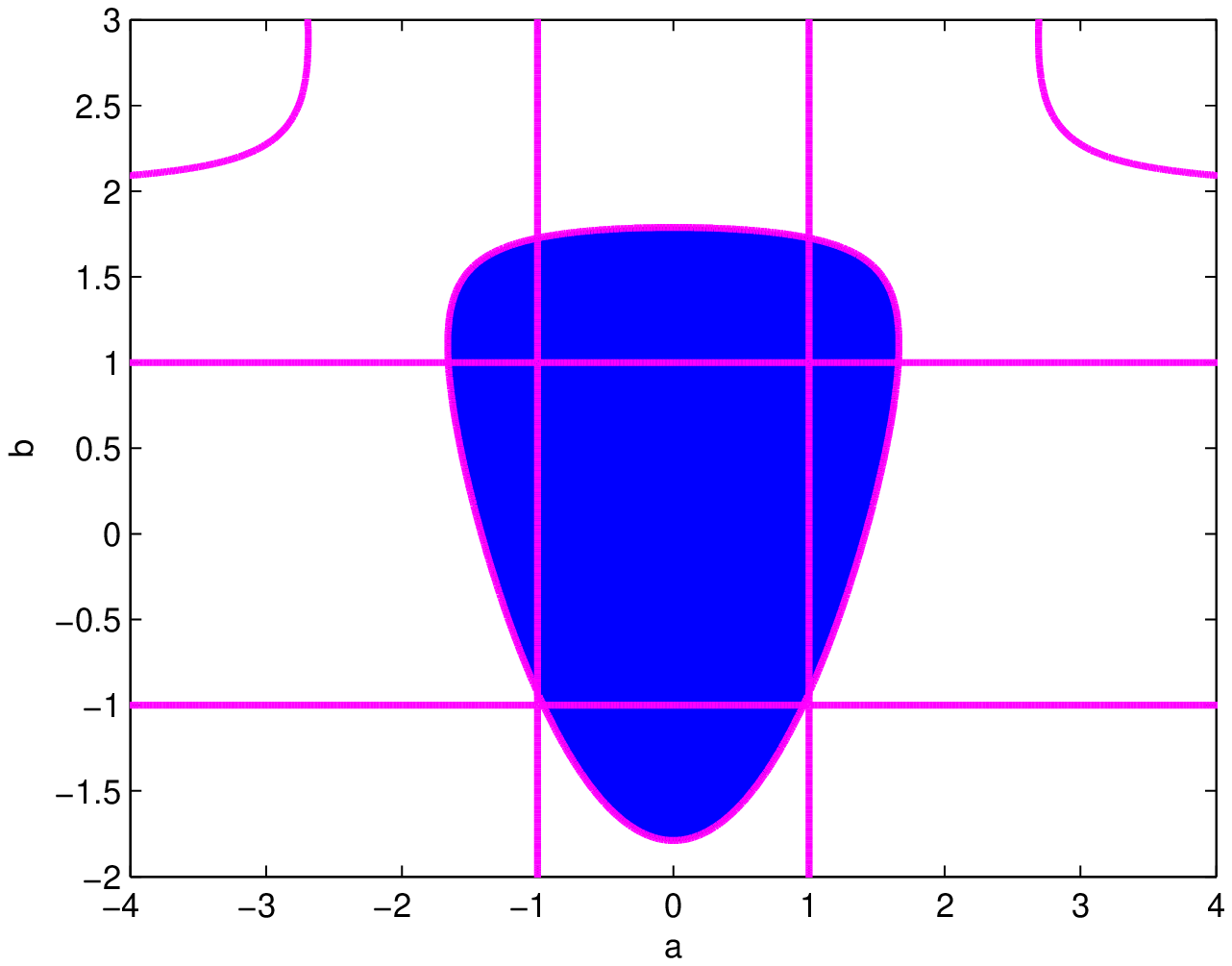} \\
\includegraphics[height=.3\textwidth]{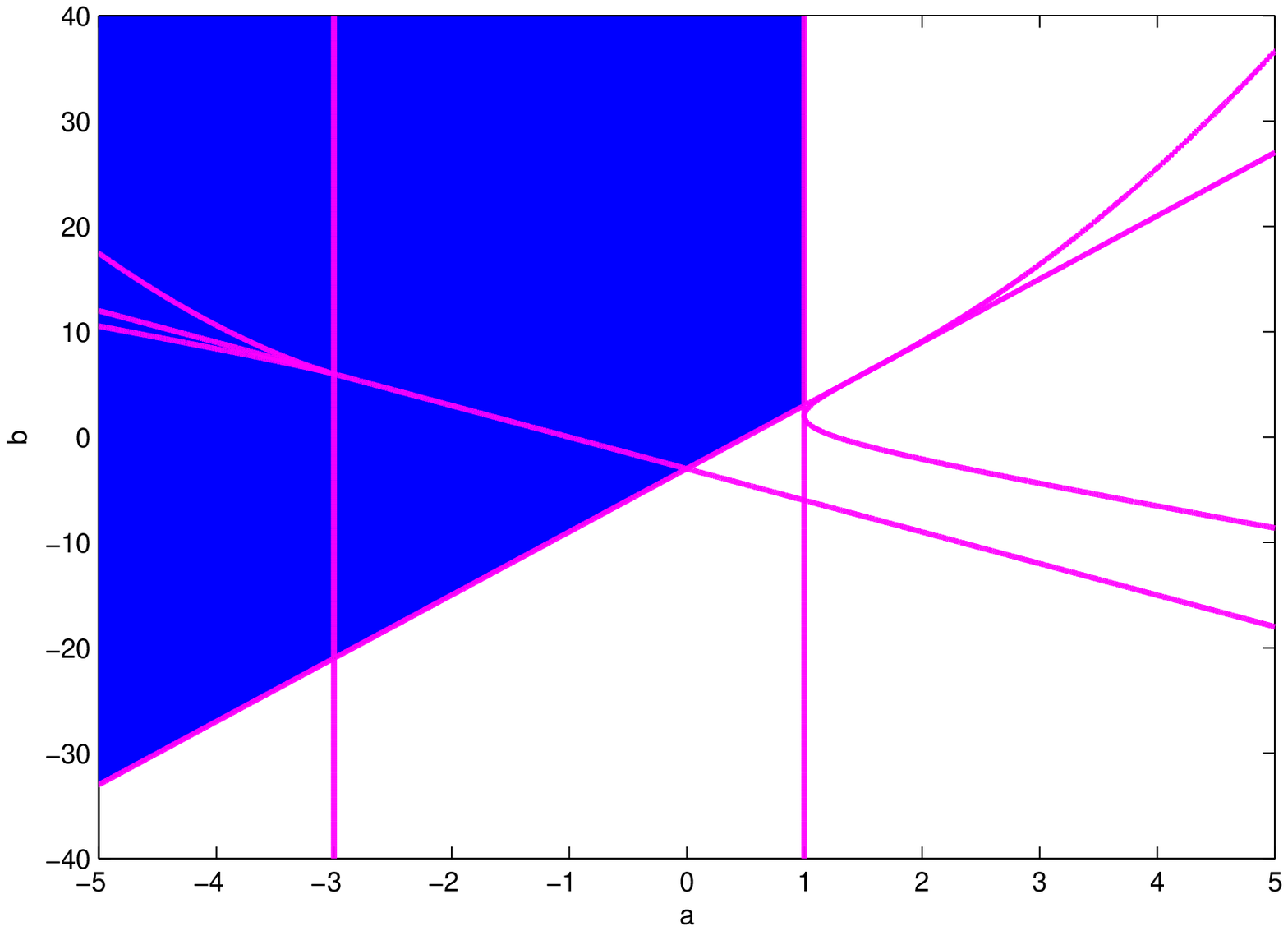} &
\includegraphics[height=.3\textwidth]{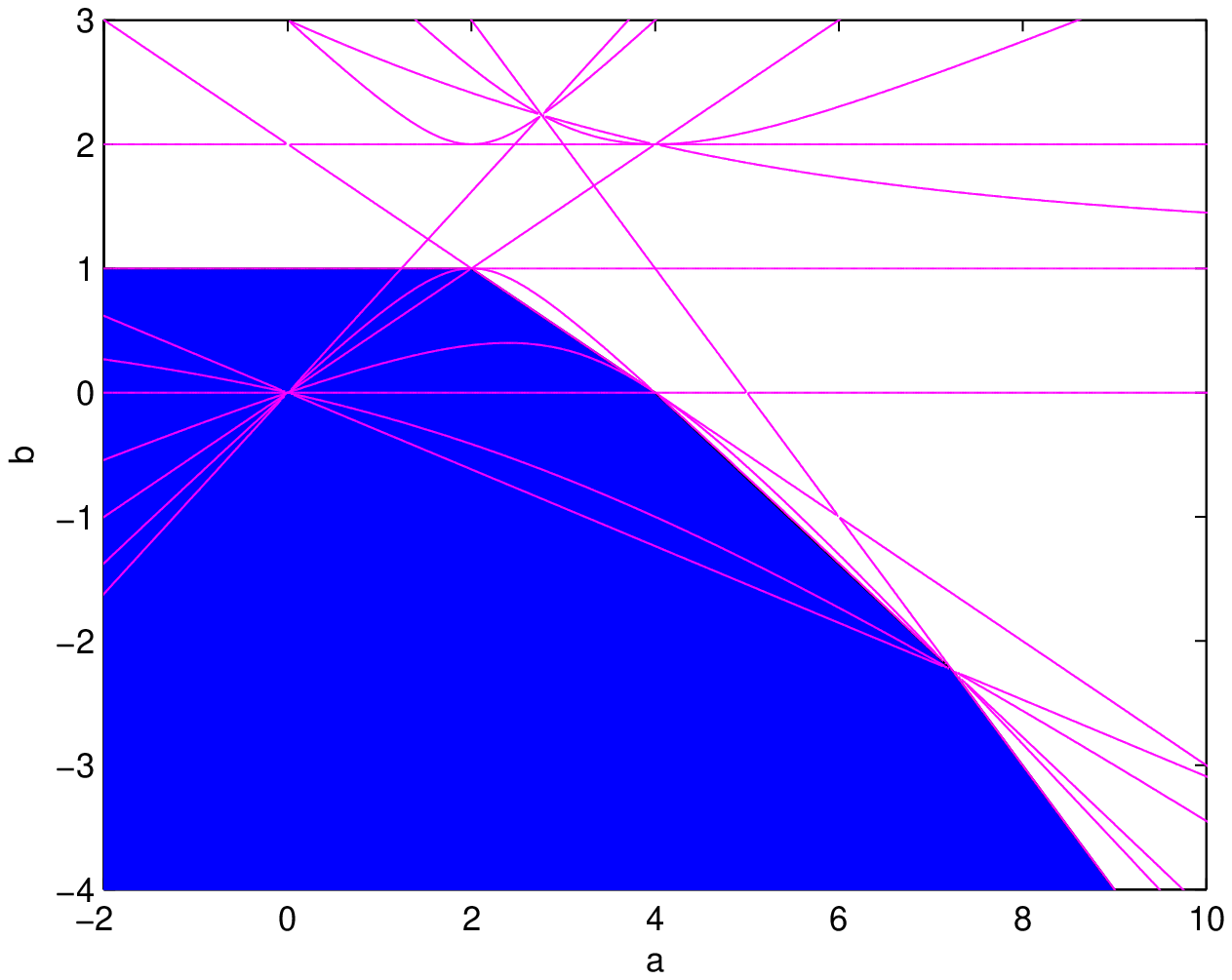}
\etab
\caption{ The pictures of curves $\varphi(a,b)=0$
and regions $F$ for polynomials $f_{a,b}(x)$ in Example~\ref{4exmp:psd-form}.
The upper left is for (i), the upper right for (ii),
the lower left for (iii), and the lower right for (iv).}
\label{fig:psd-form}
\end{figure}

\smallskip
\noindent
(iii) Consider the polynomials parameterized as
\[
f_{a,b}(x) = x_1^6 + x_2^6 + x_3^6
-a\Big(x_1^2(x_2^4+x_3^4)+x_2^2(x_3^4+x_1^4)+x_3^2(x_1^4+x_2^4)\Big)+bx_1^2x_2^2x_3^2.
\]
When $a=1, b=3$, $f_{1,3}(x)$ becomes Robinson's polynomial
that is nonnegative but not SOS (see Reznick \cite{Rez00}).
Robinson's polynomial has 10 nontrivial zeros, so $f_{1,3} \in P_{3,6}$.
Eliminating $x_2,x_3$ in \reff{eq:g=grad=0} gives $\varphi(a,b)$ as
\[
(a-1)\cdot (a+3)\cdot (3a+b+3)\cdot (6a-b-3) \cdot (2a^3+a^2b+3a^2-b^2+3b-9).
\]
%
%
The curve $\Delta(f_{a,b})=0$ lies on $\varphi(a,b)=0$. Let
\[
F = \left\{(a,b) \in \re^2: (x_1^2+x_2^2+x_3^2+x_4^2)f_{a,b} \mbox{ is SOS in } x \right\}.
\]
It is an unbounded convex set in $\re^2$.
To get the shape of $F$, we bound $a,b$ as $a+5\geq 0$, $40-b\geq 0$.
Using the method in (i), we get $F$
is the shaded area of the lower left picture in Figure~\ref{fig:psd-form}.
The curves there are defined by $\varphi(a,b)=0$. Let $G=\{(a,b): f_{a,b} \in P_{3,6}\}$.
Clearly, $F \subset G$ and the boundary of $G$ lies on $\varphi(a,b)=0$.
If $f_{a,b}(x) \in P_{3,6}$, then $f_{a,b}(1,1,1)\geq 0$ and $f_{a,b}(1,1,0)\geq 0$ imply
\[
b \geq 6a-3, \qquad a \leq 1.
\]
From the picture, we can see that $F$ is a maximal convex region
whose boundary lies on $\varphi(a,b)=0$
and satisfies the above two linear constraints. So $F=G$.

\smallskip
\noindent
(iv) Consider the polynomials parameterized as
\[
\baray{c}
f_{a,b}(x) = (x_1^2+\cdots+x_5^2)^2
- a (x_1^2x_2^2+ x_2^2x_3^2+x_3^2x_4^2+x_4^2x_5^2+ x_5^2x_1^2) \\
-b(x_1^4+x_2^4+x_3^4+x_4^4+x_5^4).
\earay
\]
When $a=4,b=0$, $f_{4,0}(x)$ becomes Horn's polynomial (see Reznick \cite{Rez00}).
Eliminating $x_2,x_3,x_4,x_5$ in \reff{eq:g=grad=0} gives $\varphi(a,b)$ as
\[
\baray{c}
(a+b-5) \cdot (a-2b)\cdot (a+2b-4)\cdot (b-1)\cdot b \cdot (b-2) \cdot (a^2+2ab-4b^2)\cdot \\
(a^2-2b^2-4a+6b)\cdot (a^2-2ab-4b^2-4a+16b)\cdot (ab+2b^2-a-6b).
\earay
\]
The curve $\Delta(f_{a,b})=0$ lies on $\varphi(a,b)=0$. Let
\[
F = \left\{(a,b) \in \re^2: (x_1^2+x_2^2+x_3^2+x_4^2+x_5^2)f_{a,b} \mbox{ is SOS in } x \right\}.
\]
It is also an unbounded convex set.
To get the shape of $F$, we bound $a,b$ as $a+2\geq 0$, $b+4\geq 0$.
Using the method in (i), we get $F$ is the shaded area of
the lower right picture in Figure~\ref{fig:psd-form}.
The curves there are defined by $\varphi(a,b)=0$.
Let $G=\{(a,b): f_{a,b} \in P_{3,6}\}$. Clearly, $F \subset G$ and the boundary of $G$ lies on $\varphi(a,b)=0$.
Then $f_{a,b}(1,0,0,0,0) \geq 0$, $f_{a,b}(1,1,0,0,0) \geq 0$, $f_{a,b}(1,1,1,1,1) \geq 0$
imply that any pair $(a,b)\in G$ satisfies
\[
a+b-5 \leq 0, \quad a+2b -4 \leq 0, \quad b-1 \leq 0.
\]
Since $(3.10,0.5), (5.5,-1), (9.1,-4) \notin G$
(verified by software {\it GloptiPoly~3} \cite{GloPol3}),
by observing the lower right picture in Figure~\ref{fig:psd-form},
we can see that $F$ is a maximal convex region that satisfies the above three linear constraints,
excludes the previous 3 pairs, and has the boundary lying on $\varphi(a,b)=0$. So $F=G$.
\qed
\end{exm}

\subsection{Nonnegative multihomogeneous forms}

In this subsection, we study the cone of nonnegative multihomogeneous forms.
Let $M_{d_1,\ldots,d_r}^{n_1,\ldots,n_r}$ denote the space
of multihomogeneous forms in the space $\re^{n_1} \times \cdots \times \re^{n_r}$
which are homogeneous of degree $d_i$ in each $\re^{d_i}$.
Thus every $f \in M_{d_1,\ldots,d_r}^{n_1,\ldots,n_r}$ has the form
\[
f =  \sum_{ (\af_1,\ldots,\af_r) \in \N^{n_1} \times \cdots \times \N^{n_r} }
f_{\af_1,\ldots,\af_r}  (x^{(1)})^{\af_1} \cdots  (x^{(r)})^{\af_r}.
\]
Here we assume all the degrees $d_i$ are even.
Let $P_{d_1,\ldots,d_r}^{n_1,\ldots,n_r}$ be the cone of forms in
$M_{d_1,\ldots,d_r}^{n_1,\ldots,n_r}$ that are nonnegative everywhere.

Given $f \in M_{d_1,\ldots,d_r}^{n_1,\ldots,n_r}$,
we say $(u^{(1)},\ldots,u^{(r)}) \in \prod_{i=1}^r \cpx^{n_i}$
is a critical point of $f$ in $\prod_{i=1}^r \P^{n_i-1}$
if every $u^{(i)} \ne 0$ and
\[
\nabla_{x^{(1)}} f(u^{(1)},\ldots,u^{(r)}) = 0, \quad \ldots, \quad
\nabla_{x^{(r)}} f(u^{(1)},\ldots,u^{(r)}) = 0.
\]
Let $H_{d_1,\ldots,d_r}^{n_1,\ldots,n_r} \subset M_{d_1,\ldots,d_r}^{n_1,\ldots,n_r}$ be the set
\[
H_{d_1,\ldots,d_r}^{n_1,\ldots,n_r} = \left\{ f \in M_{d_1,\ldots,d_r}^{n_1,\ldots,n_r}:
f \mbox{ has a critical point in } \prod_{i=1}^r \P^{n_i-1} \right\}.
\]
It was shown in \cite[Prop.~2.3 in Chap.13]{GKZ} that
$H_{d_1,\ldots,d_r}^{n_1,\ldots,n_r}$ is a hypersurface if and only if
\be  \label{mdis:hpsurf-cond}
2(n_i-1) \leq n_1+\cdots+n_r-r \quad
\mbox{ for all $i$: } \, d_i = 1.
\ee
In particular, if every $d_i>1$, $H_{d_1,\ldots,d_r}^{n_1,\ldots,n_r}$
is a hypersurface for any dimensions $n_1,\ldots, n_r$.
When \reff{mdis:hpsurf-cond} holds,
we still denote by $\Delta(f)$ a defining polynomial of
the lowest degree for $H_{d_1,\ldots,d_r}^{n_1,\ldots,n_r}$.
It can be chosen to have coprime integer coefficients and is unique up to a sign.
The polynomial $\Delta(f)$ is also called the discriminant of the multihomogeneous form $f$.

\begin{theorem}
When all $d_i>0$ are even,
the boundary $\pt P_{d_1,\ldots,d_r}^{n_1,\ldots,n_r}$
lies on the hypersurface $H_{d_1,\ldots,d_r}^{n_1,\ldots,n_r}$
whose degree is the coefficient of the term $z_1^{n_1-1}\cdots z_r^{n_r-1}$
in the power series expansion of the following rational function
\[
\left( \prod_{j=1}^r(1+z_j) \left(1 -
\sum_{j=1}^r \frac{d_jz_j}{(1+z_j)} \right) \right)^{-2}.
\]
\end{theorem}
\begin{proof}
Since all $d_i>0$ are even, the condition \reff{mdis:hpsurf-cond} holds,
and $H_{d_1,\ldots,d_r}^{n_1,\ldots,n_r}$ is a hypersurface defined by $\Delta(f)=0$.
A multihomogeneous form $f \in P_{d_1,\ldots,d_r}^{n_1,\ldots,n_r}$ if and only if
\[
\lmd_{min}(f) := \min_{\|x^{(1)}\|_2 = \cdots = \|x^{(r)}\|_2 = 1}
f(x^{(1)}, \ldots , x^{(r)}) \quad \geq \quad 0.
\]
Clearly, $f\in\pt P_{d_1,\ldots,d_r}^{n_1,\ldots,n_r}$
if and only if $\lmd_{min}(f) = 0$. If $f \in \pt P_{d_1,\ldots,d_r}^{n_1,\ldots,n_r}$,
then we can find $u^{(1)},\ldots,u^{(r)}$ of unit length satisfying
$f(u^{(1)},\ldots,u^{(r)})=0$ and
\[
\nabla_{x^{(1)}} f(u^{(1)},\ldots,u^{(r)}) = 0, \quad \ldots, \quad
\nabla_{x^{(r)}} f(u^{(1)},\ldots,u^{(r)}) = 0.
\]
Thus, $f$ also belongs to $H_{d_1,\ldots,d_r}^{n_1,\ldots,n_r}$.
The degree formula for $H_{d_1,\ldots,d_r}^{n_1,\ldots,n_r}$
is given by Theorem~2.4 of Chapter~13 in \cite{GKZ}.
\end{proof}

\begin{figure}[htb]
\centering
\btab{cc}
\includegraphics[width=.6\textwidth]{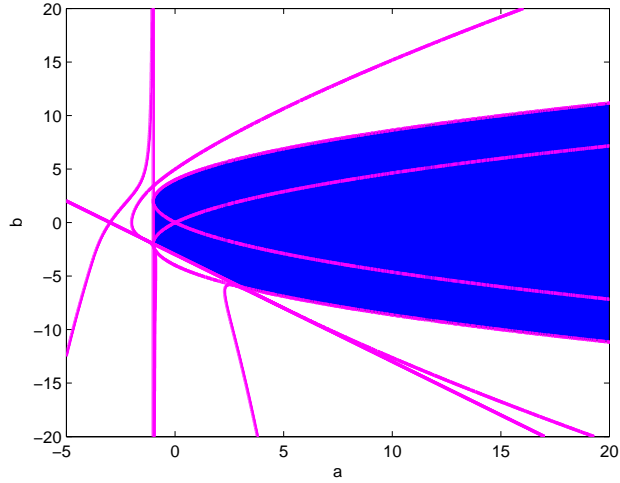}
\etab
\caption{The picture of curve $\varphi(a,b)=0$
and region $F$ for bi-quadratic forms $f_{a,b}(x)$ in Example~\ref{exmp-biQ:3-3}.}
\label{fig:biQ-var6deg4}
\end{figure}

\begin{exm}  \label{exmp-biQ:3-3}
Consider the bi-quadratic forms parameterized as
\[
\baray{c}
f_{a,b}(x) = (x_1^2+x_2^2+x_3^2)(x_4^2+x_5^2+x_6^2)
+a(x_1^2x_5^2+x_2^2x_6^2+x_3^2x_4^2) \\
+b(x_1x_2x_4x_5+x_1x_3x_4x_6+x_2x_3x_5x_6).
\earay
\]
Here $n_1=n_2=3,d_1=d_2=2$.
First, we dehomogenize $f_{a,b}(x)$ as $g = f_{a,b}(1,x_2,x_3,x_4, 1, x_6)$,
and then use the function {\it elim} in {\it Singular}
to determine all pairs $(a,b)$ satisfying $\Delta(f_{a,b}) = 0$.
Eliminating $x_2,x_3,x_4,x_6$ from
\[
 g = \frac{\pt g}{\pt x_2} = \frac{\pt g}{\pt x_3} =
\frac{\pt g}{\pt x_4} = \frac{\pt g}{\pt x_6} =0
\]
gives the equation $\varphi(a,b)=0$ where $\varphi(a,b)$ is
\[
\baray{l}
(a+1)\cdot (a+b+3) \cdot (a^2-ab+b^2) \cdot (-b^2+4a+4b) \cdot (-b^2+4a-4b) \cdot \\
(a^3b^4-16a^6-8a^4b^2-4a^3b^3+3a^2b^4+ab^5-80a^5-16a^4b-32a^3b^2-20a^2b^3 \\
-4ab^4+b^5-96a^4-32a^3b-24a^2b^2-12ab^3-5b^4).
\earay
\]
The curve $\Delta(f_{a,b}) = 0$ lies on $\varphi(a,b)=0$. Let
\[
F = \left\{(a,b) \in \re^2: (1+x_2^2+x_3^2+x_4^2+x_6^2) \cdot
f_{a,b}(1,x_2,x_3,x_4,1,x_6) \mbox{ is SOS } \right\}.
\]
By the method used in Example~\ref{4exmp:psd-form},
$F$ is drawn in the shaded area of Figure~\ref{fig:biQ-var6deg4}.
The curves there are defined by $\varphi(a,b)=0$.
Let $G=\left\{(a,b):f_{a,b}(x) \in P_{2,2}^{3,3}\right\}$.
Clearly, $F\subset G$ and the boundary of $G$ lies on $\varphi(a,b)=0$.
If $f_{a,b}(x) \in P_{2,2}^{3,3}$, then from
\[
f_{a,b}(1,1,1,1,1,1)\geq 0, \quad f_{a,b}(1,0,0,1,0,0)\geq 0
\]
we know every $(a,b)\in G$ satisfies
\[
a+b+3 \geq 0, \qquad a+1 \geq 0.
\]
Because $f_{20,15}(x) \not\in G$
($\because \nabla_{x_1,x_2,x_3}^2 f_{20,15}$ has negative eigenvalue at $(1,1,0)$)
and $f_{20,-15}(x) \not\in G$ ($\because \nabla_{x_1,x_2,x_3}^2 f_{20,-15}$
has negative eigenvalue at $(1,-1,0)$), from the picture we can see that $F$ is a maximal convex region
that excludes $(20,15)$ and $(20,-15)$, satisfies the above two linear constraints,
and has boundary lying on $\varphi(a,b)=0$. So $F=G$.
\qed
\end{exm}

\section{Polynomials nonnegative on a variety} \label{sec:real-var}
\setcounter{equation}{0}

This section studies the cone $P_d(K)$ when $K$ is a real algebraic variety defined as
\[
K = \{ x \in \re^n: \, g_1(x) = \cdots = g_m(x) =0 \}.
\]
Here $g=(g_1,\ldots,g_m)$ is a tuple of polynomials. For convenience, denote $P_d(K)$ as
\[
P_d(g) = \big\{ f(x) \in \re[x]_{\leq d}: \,
f(x) \geq 0 \text{ for every } x \in V_\re(g) \big\}.
\]

To study the boundary $\pt P_d(g)$ of $P_d(g)$, we need a characterization for it.
One would think if $f$ lies on $\pt P_d(g)$
then $f(x)$ vanishes somewhere on $V_\re(g)$. However, this is not always true.
For a counterexample, consider $f=x_1+x_2$ and $g=x_1^3+x_2^3-1$.
Clearly, $f$ is strictly positive on $V_\re(g)$, but it lies on $\pt P_1(g)$.
For any $\eps>0$ the polynomial $x_1+x_2-\eps$ is no longer nonnegative on $V_\re(g)$ because
\[
\inf_{ x \in V_\re(g) } \quad x_1+x_2 \quad = \quad 0.
\]
The reason is that $V_\re(g)$ is not compact.
We need other characterization in this case.

Let $V_\re^h(g)$ be the homogenization of $V_\re(g)$, that is,
\[
V_\re^h(g) = \left\{ \tilde{x}\in \re^{n+1}: \,
  g_1^h(\tilde{x}) = \cdots = g_m^h(\tilde{x}) = 0 \right\}.
\]
Clearly, if $f^h$ is nonnegative on $V_\re^h(g)$,
then $f$ is also nonnegative on $V_\re(g)$, but the reverse is not necessarily true.
For this purpose, we need a new condition.
We say the variety $V_\re^h(g)$ is {\it closed at $\infty$} if
\[
V_\re^h(g) \cap \{x_0 \geq 0 \} = \mbox{closure}
\left( V_\re^h(g) \cap \{x_0 > 0 \} \right).
\]
Define two constants
\begin{align} \label{def:dtgf}
\delta_g(f):= & \min_{ x \in V_\re(g) }  \quad f(x), \\
\delta_g^h(f):= & \min_{\tilde{x} \in V_\re^h(g): \|\tilde{x}\|_2=1, \, x_0 \geq 0 }
 \quad f^h(\tilde{x}).    \label{dth_gf}
\end{align}

The boundary $\pt P_d(g)$ is characterized as below.

\begin{prop}  \label{char-bdry:Pd(g)}
Let $g$ be given as above.
\bit
\item [(i)] If $V_\re(g)$ is compact, then
\[
\delta_g(f) > 0   \, \Leftrightarrow \, f \in int\big(P_d(g)\big),
\quad \mbox{ and } \quad
\delta_g(f) = 0  \, \Leftrightarrow \, f \in  \pt P_d(g).
\]

\item [(ii)] If $V_\re^h(g)$ is closed at $\infty$, then
\[
\delta_g^h(f) > 0   \, \Leftrightarrow \, f \in int\big(P_d(g)\big),
\quad \mbox{ and } \quad
\delta_g^h(f) = 0  \, \Leftrightarrow \, f \in  \pt P_d(g).
\]
\eit
\end{prop}
\begin{proof}
Part (i) is quite clear. We prove part (ii).
For any $\tilde{u} \in V_\re^h(g)$ with $u_0 \geq 0$,
we can find a sequence $(t_k, w_k) \in V_\re^h(g)$ with every $t_k>0$
approaching $\tilde{u}$. Note that $w_k/t_k \in V_\re(g)$.
So, if $f \in P_d(g)$, then
\[
f^h(\tilde{u}) = \lim_{ k \to \infty} f^h(t_k, w_k) = \lim_{ k \to \infty} t_k^d f(w_k/t_k) \geq 0,
\]
and we have $\delta_g^h(f) \geq 0$.
On the other hand, if $\delta_g^h(f) \geq 0$, then for every $v\in V_\re(g)$
\[
f(v) = f^h(1,v)  = (1+\|v\|_2^2)^{d/2} f^h\left((1,v)/(1+\|v\|_2^2)^{1/2}\right)
\geq (1+\|v\|_2^2)^{d/2} \delta_g^h(f) \geq 0,
\]
and we get $f\in P_d(g)$.
The above implies  $\delta_g^h(f) \geq 0$ if and only if $f\in P_d(g)$.

By definition, $\delta_g^h(f)$ is the minimum of a polynomial function over a compact set.
If $\delta_g^h(f) >0$, then in a small neighborhood $\mc{O}$ of $f$
we have $\delta_g^h(p) >0$ for every $ p \in \mc{O}$, that is,
$f$ lies in the interior of $P_d(g)$.
If $\delta_g^h(f) =0$, then we can find $p\in \re[x]_{\leq d}$
of arbitrarily small coefficients such that $\delta_g^h(f+p) <0$, that is, $f\in \pt P_d(g)$.
\end{proof}

We would like to remark that not every $V_\re^h(g)$ is closed at $\infty$,
and even if $V_\re(g)$ is compact $V_\re^h(g)$ might still not be closed at $\infty$.

\begin{exm}
(i) Let $g=x_1^2(x_1-x_2)-1$ and $f=x_1-x_2+1$.
The polynomial $f$ is strictly positive on the variety $V_\re(g)$,
but $f^h = x_1-x_2+x_0$ is not nonnegative on
\[
V_\re^h(g) = \left\{ (x_0,x_1,x_2) \in \re^3:  x_1^2(x_1-x_2)- x_0^3 =0 \right \}.
\]
This is because $(0,0,1) \in V_\re^h(g)$ while $f^h(0,0,1)<0$.
So $V_\re^h(g)$ is not closed at $\infty$.

\smallskip
\noindent
(ii) Let $g=x_1^2(1-x_1^2-x_2^2)-x_2^2$. The variety $V_\re(g)$ is compact. Its homogenization is
\[
V_\re^h(g) = \left\{ \tilde{x}:  x_1^2(x_0^2-x_1^2-x_2^2)-x_0^2x_2^2 =0 \right\}.
\]
However, $V_\re^h(g)$ is not closed at $\infty$. Otherwise, for every
$\tilde{u} \in V_\re^h(g) \cap \{x_0 = 0 \}$ we have
\[
\tilde{u} = \lim_{t_k >0, \, t_k \to 0} t_k (1, v_k) \quad \mbox{ for some } \quad v_k \in V_\re(g).
\]
This implies $V_\re^h(g) \cap \{x_0 = 0 \}$ is compact, which is clearly false.
\qed
\end{exm}

Now we study the boundary of the cone $P_d(g)$.

\begin{theorem}  \label{psd-var:bdry}
Let $g=(g_1,\ldots,g_m)$ be given as above, and $\deg(g_i) = d_i$.
Suppose $m\leq n$.

\bit

\item [(i)] If $V_\re(g) \ne \emptyset$, and either $V_\re(g)$ is compact
or $V_\re^h(g)$ is closed at $\infty$,
then the boundary $\pt P_d(g)$ lies on the hypersurface
\[
\mathcal{E}_d(g) = \{f \in \re[x]_{\leq d}: \Delta(f,g_1,\ldots,g_m) = 0 \}.
\]

\item [(ii)] If the projective variety $V_\P(g_1^h,\ldots, g_m^h)$ is nonsingular,
the degree of $\mc{E}_d(g)$ is
\be \label{eq-deg:bdrP(g)}
\left( \prod_{i=1}^m d_i \right) \cdot
S_{n-m}\Big(d-1,d-1,d_1-1, \ldots, d_m-1\Big).
\ee
Otherwise, the above is only an upper bound.

\item [(iii)] The polynomial $\Delta(f,g_1,\ldots,g_m)$ is identically zero in $f$
if and only if the projective variety $V_\P(g_1^h,\ldots, g_m^h)$
has a positive dimensional singular locus.

\eit

\end{theorem}

\begin{proof}
(i) We first consider the case that $V_\re^h(g)$ is closed at $\infty$.
Let $f(x) \in \pt P_d(g)$. By Proposition~\ref{char-bdry:Pd(g)},
we know $f^h$ is nonnegative on $V_\re^h(g)$
and vanishes at some $ 0 \ne \tilde{u} \in V_\re^h(g)$.
So $\tilde{u}$ is a minimizer of $f^h(\tilde{x})$ on $V_\re^h(g)$.
By Fritz-John optimality condition (see Sec. 3.3.5 in \cite{Bsks}),
there exists $(\mu_0, \mu_1,\ldots,\mu_m)\ne 0$ satisfying
\[
\baray{c}
\mu_0 \nabla_{\tilde{x}} f_0(\tilde{u}) + \mu_1 \nabla_{\tilde{x}} g_1(\tilde{u})
+ \cdots + \mu_m \nabla_{\tilde{x}} g_m(\tilde{u}) = 0, \\
f(\tilde{u}) = g_1(\tilde{u}) = \cdots = g_m(\tilde{u})=0.
\earay
\]
By relation \reff{dis:f0m}, we know $\Delta(f,g_1,\ldots,g_m) = 0$.

The proof for the case that $V_\re(g)$ is compact
is almost the same as the above,
and is omitted here.

(ii) When  $V_\P(g_1^h,\ldots, g_m^h)$ is nonsingular,
from the proof of part b) in Theorem~\ref{thm-deg:dis-itsc},
we know the degree of $\Delta(f,g_1,\ldots,g_m)$ in $f$ is given by \reff{eq-deg:bdrP(g)}.
When  $V_\P(g_1^h,\ldots, g_m^h)$ is singular, the formula in \reff{eq-deg:bdrP(g)}
is only an upper bound by perturbing the coefficients of $g_1,\ldots,g_m$.

(iii) This immediately follows part c) of Theorem~\ref{thm-deg:dis-itsc}.
\end{proof}

We have seen that there is no log-polynomial type barrier function
for the cone $P_d(\re^n)$ when $d>2$ and $n\geq 1$.
There is a similar result for $P_d(g)$.

\begin{theorem}
Suppose $V_\re(g)$ is nonempty, either $V_\re(g)$ is compact or $V_\re^h(g)$ is closed at $\infty$,
$V_\P(g^h)$ has positive dimension, and $d>2$ is even.
If the discriminant $\Delta(f,g_1,\ldots,g_m)$ is irreducible in $f$ over $\cpx$,
then there is no polynomial $\varphi(f)$ satisfying
\bit

\item $\varphi(f)>0$ whenever $f$ lies in the interior of $P_d(g)$, and

\item $\varphi(f)=0$ whenever $f$ lies on the boundary of $P_d(g)$.

\eit
Therefore, $-\log \varphi(f)$ can not be a barrier function
for the cone $P_d(g)$ when we require $\varphi(f)$ to be a polynomial,
and $P_d(g)$ is not representable by LMI.
\end{theorem}
\begin{proof}
We prove the first part by contradiction. Suppose such a $\varphi$ exists.
By Theorem~\ref{psd-var:bdry}, we know $\pt P_d(g)$
lies on the hypersurface $\Delta(f,g_1,\ldots,g_m)=0$.
Since $\Delta(f,g_1,\ldots,g_m)$ is irreducible in $f$,
the hypersurface $\Delta(f,g_1,\ldots,g_m)=0$ is irreducible and
equals the Zariski closure of $\pt P_d(g)$
(it is contained in some hypersurface).
Hence, the hypersurface $\varphi(f)=0$ contains $\Delta(f,g_1,\ldots,g_m)=0$,
and $\varphi(f)$ vanishes whenever $\Delta(f,g_1,\ldots,g_m)=0$.
By Hilbert's Nullstenllensatz (see Theorem~\ref{HilNull}),
there exist an integer $k>0$ and a polynomial $p(f)$ such that
\[
\varphi(f)^k = \Delta(f,g_1,\ldots,g_m) \cdot p(f).
\]
Set $\hat{f}(x) = (1+x_1^2+\cdots + x_n^2)^{d/2}$,
then $\hat{f}^h(x) = (x_0^2+x_1^2+\cdots + x_n^2)^{d/2}$.
Clearly, $\hat{f}$ lies in the interior of $P_d(g)$.
However, since $V_\P(g^h)$ has positive dimension, we know
\[
V_\P(\hat{f}^h, g^h) \,=\,
\left\{ \tilde{x} \in \P^n:  x_0^2+x_1^2+\cdots + x_n^2 =0 \right\} \cap V(g^h) \ne \emptyset
\]
by B\'{e}zout's theorem. For any $\tilde{u} \in V_\P(\hat{f}^h, g^h)$,
we have $\nabla_{\tilde{x}} \hat{f}^h(\tilde{u})=0$ ($d>2$)
which results in $\Delta(\hat{f},g_1,\ldots,g_m)=0$.
So $\varphi(\hat{f})=0$, which contradicts the first item.

The second part is a consequence of the first part,
as in the proof of Theorem~\ref{psd:nobarfun}.
\end{proof}

\subsection{Computing the discriminantal variety $\Delta(f,g_1,\ldots,g_m)=0$}
\label{subsec:dis-fg1m}

Now we discuss the connection between $\Delta(f,g_1,\ldots,g_m)$ and
the discriminant of the Lagrangian polynomial in $(x,\lmd)$
\[
L(x,\lmd) =   f(x) + \sum_{i=1}^k \lmd_i g_i(x).
\]
When $V_\re(g)$ is compact, $f \in \pt P_d(g)$ if and only if
$\delta_g(f)=0$, i.e., there exists $u \in V_{\re}(g)$
such that $f(u)=0$ and $u$ is a minimizer of $f$ on $V_{\re}(g)$.
So, if $f(x) \in \pt P_d(g)$ and $V_\re(g)$ is nonsingular at $u$,
the Karush-Kuhn-Tucker (KKT) condition (see Sec.~3.3 in \cite{Bsks})holds,
and there exists $\mu = (\mu_1,\ldots,\mu_m)$ satisfying
\[
\nabla_{x} f(u) +
\overset{m}{\underset{i=1}{\sum} } \mu_i \nabla_{x} g_i(u) = 0, \quad
g_1(u) = \cdots = g_m(u)=0.
\]
The above is equivalent to that
$(u,\mu)$ is a critical zero point of $L(x,\lmd)$, that is,
\[
\nabla_{x,\lmd}L(u,\mu) =0, \quad  L(u,\mu) = 0.
\]
Hence, we have $\Delta(L) = 0$. Therefore, the hypersurface $\Delta(f,g_1,\ldots,g_m)=0 $
would be possibly determined via investigating $\Delta(L)=0$.
To the best knowledge of the author, no general procedure is known in computing
the discriminant of type $\Delta(f,g_1,\ldots,g_m)$.
Though there exist systemic methods for evaluating $\Delta(L)$,
its computation and formula would be too complicated to be practical,
as we have seen in the preceding section.
In the following, we propose a different approach using elimination.

Suppose $f = f(x;p)$ is a polynomial in $x$ whose coefficients
are also polynomial in a parameter $p=(a,b,c,\ldots)$
over the rational field, i.e., from the ring $\Q[p]$.
So, if $f(x;p) \in \pt P_d(g)$ and $V_\re(g)$ is a nonsingular compact set,
then $f$ satisfies the over-determined polynomial system in $(x,\lmd)$
\be \label{cond:KKT}
\left.
\baray{r}
\nabla_{x} f(x) +
\overset{m}{ \underset{i=1}{\sum} } \lmd_i \nabla_{x} g_i(x) = 0 \\
f(x) = g_1(x) = \cdots = g_m(x)=0
\earay \right\}.
\ee
The equation that $p$ satisfies would be determined by eliminating $(x,\lmd)$ in the above.
Let $\varphi(p)=0$ be the polynomial equation obtained by eliminating $(x,\lmd)$ in \reff{cond:KKT}.
So, if $p$ satisfies $\Delta(f,g_1,\ldots,g_m)=0$, then $\varphi(p) = 0$.
Computing $\varphi(p)$ would be done by using {\it elim} in ${\it Singular}$ \cite{Sing}.
We illustrate this in the below.

\begin{exm}
Consider the polynomials parameterized as
\[
f= x_1^2+ax_1x_2+bx_1+cx_2+d,
\]
and $K=\{ x_1^2+x_2^2=1\}$ is a circle.
The polynomial $\varphi(a,b,c)$ obtained by eliminating $(x,\lmd)$ in \reff{cond:KKT} is
\[
\baray{l}
a^6-3a^4b^2+3a^2b^4-b^6-3a^4c^2-21a^2b^2c^2-3b^4c^2+3a^2c^4-3b^2c^4-c^6\\
+36a^3bcd+18ab^3cd+18abc^3d-8a^4d^2-20a^2b^2d^2+b^4d^2-20a^2c^2d^2\\
+2b^2c^2d^2+c^4d^2-16abcd^3+16a^2d^4+18a^3bc-18ab^3c+36abc^3-8a^4d\\
-2a^2b^2d+10b^4d-38a^2c^2d+2b^2c^2d-8c^4d-24abcd^2+32a^2d^3-8b^2d^3\\
+8c^2d^3+a^4-2a^2b^2+b^4-20a^2c^2+20b^2c^2-8c^4+24abcd+8a^2d^2-32b^2d^2\\
-8c^2d^2+16d^4+16abc-8a^2d-8b^2d-32c^2d+32d^3-16c^2+16d^2.
\earay
\]
It is a polynomial of degree $6$ in $4$ variables.
The set $\{(a,b,c): f \in \pt P_2(K)\}$
lies on the surface $\varphi(a,b,c) = 0$.
\qed
\end{exm}

\begin{exm}  \label{psd-deg4:2circ}
(i) Consider the polynomials parameterized as
\[
f=x_1^4+ax_1^3x_2+bx_1x_2^3+c,
\]
and $K=\{ x_1^2+x_2^2=1\}$ is a circle.
The polynomial $\varphi(a,b,c)$ obtained by eliminating $(x,\lmd)$ in \reff{cond:KKT} is
\[
\baray{l}
4a^3b^3+27a^4c^2-36a^3bc^2+2a^2b^2c^2-36ab^3c^2+27b^4c^2-256a^2c^4\\
+512abc^4-256b^2c^4+6a^2b^2c-36ab^3c+54b^4c-288a^2c^3+704abc^3\\
-544b^2c^3+27b^4+192abc^2-288b^2c^2-256c^4-256c^3.
\earay
\]
The surface $\varphi(a,b,c)=0$
is drawn in the left picture in Figure~\ref{fig:deg4-2circ}.
It contains the set $\{(a,b,c): f \in \pt P_4(K)\}$.

\begin{figure}[htb]
\centering
\btab{cc}
\includegraphics[width=.45\textwidth]{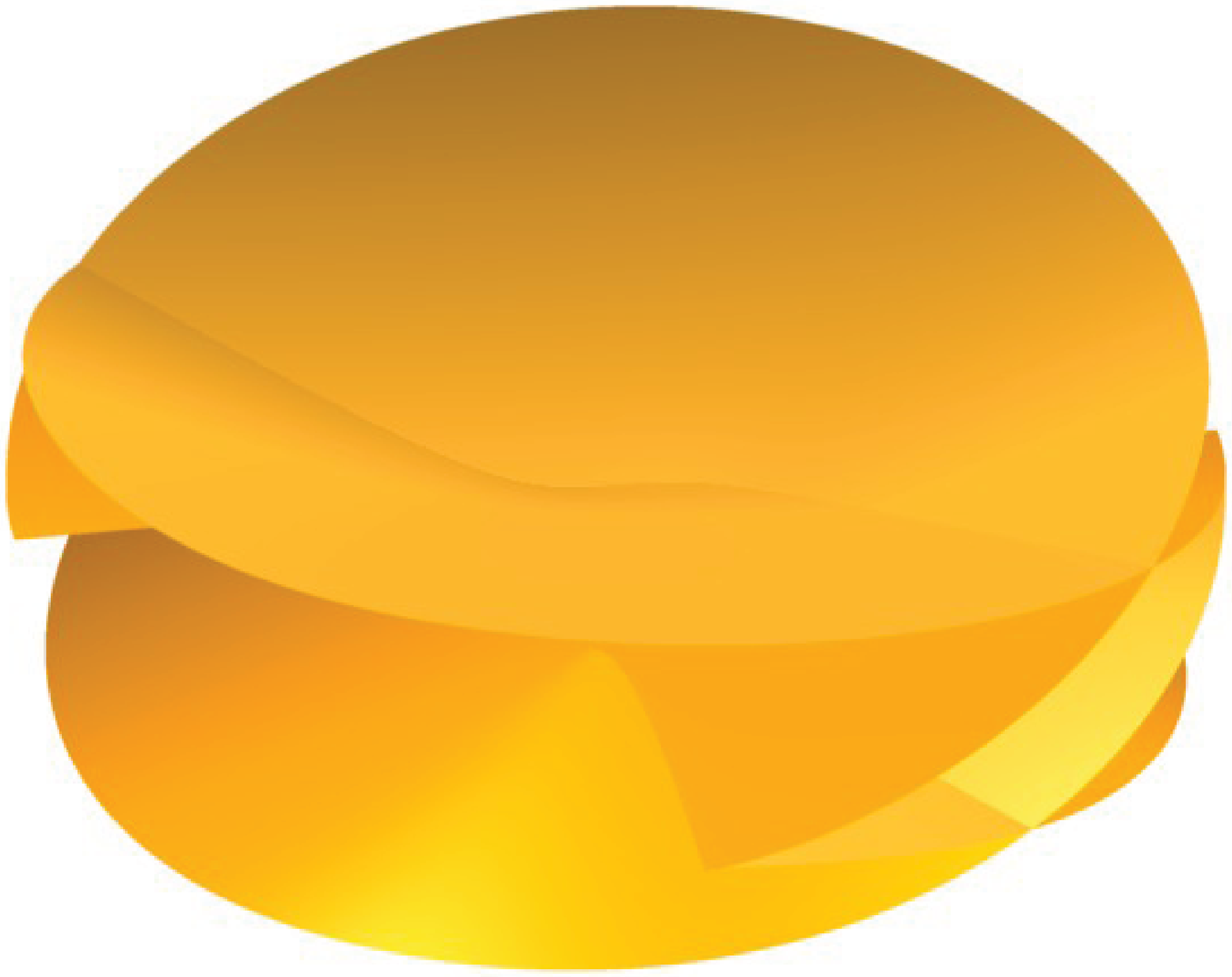} &
\includegraphics[width=.45\textwidth]{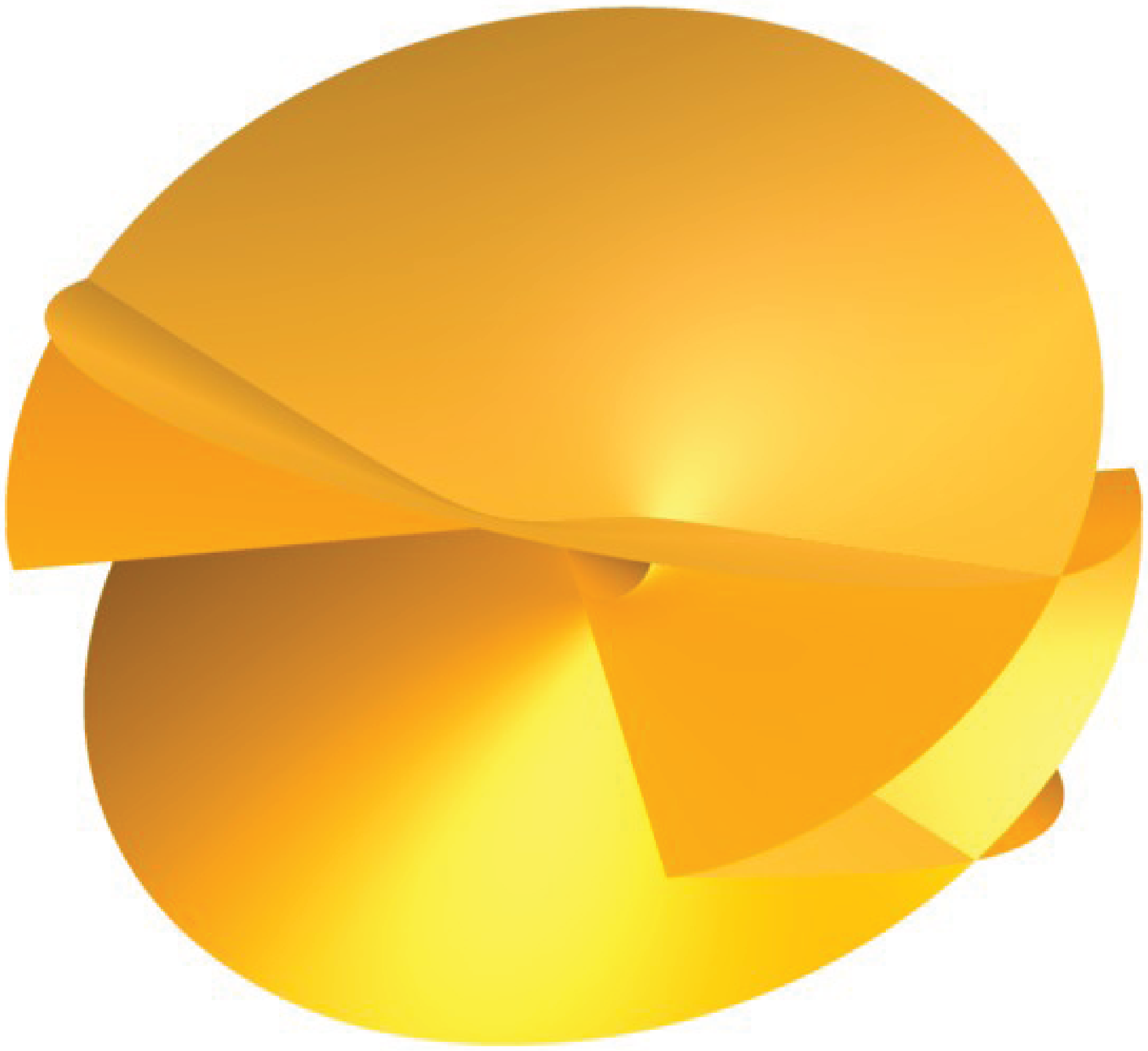}
\etab
\caption{ The pictures of surfaces $\varphi(a,b,c)=0$ in Example~\ref{psd-deg4:2circ}.
The left is for (i), and the right is for (ii).}
\label{fig:deg4-2circ}
\end{figure}

\smallskip
\noindent
(ii) Consider the polynomials parameterized as
\[
f =x_1^4+ax_1^3x_2+bx_1x_2^3+c,
\]
and $K=\{ x_1^4+x_2^4=1\}$ is a circle defined in $4$-norm.
The polynomial $\varphi(a,b,c)$ obtained by eliminating $(x,\lmd)$ in \reff{cond:KKT} is
\[
\baray{c}
4a^3b^3+27a^4c^2+6a^2b^2c^2+27b^4c^2+192abc^4-256c^6+6a^2b^2c\\
+54b^4c+384abc^3-768c^5+27b^4+192abc^2-768c^4-256c^3.
\earay
\]
The surface $\varphi(a,b,c)=0$
is drawn in the right picture in Figure~\ref{fig:deg4-2circ}.
It contains the set $\{(a,b,c): f \in \pt P_4(K)\}$.

The surfaces in Figure~\ref{fig:deg4-2circ} are drawn by Labs' software {\it Surfex}
which is downloaded from the website \url{www.surfex.algebraicsurface.net}.
\qed
\end{exm}

\subsection{Resolution of singularities}

In Theorem~\ref{psd-var:bdry}, we know if
the projective variety $V_\P(g^h)$ has a positive dimensional singular locus,
then $\Delta(f,g_1,\ldots,g_m)$ is identically zero in $f$
and $\Delta(f,g_1,\ldots,g_m) = 0$ defines the whole space $\re^n$.
This is not what we want, because the boundary $\pt P_d(g)$ usually has codimension one.
To study $\pt P_d(g)$, we need to resolve the singularities of $V_\P(g^h)$.
By Hironaka's result (see Theorem 17.23 in Harris's book \cite{Ha}),
there exist a smooth projective variety $U \subset \P^n$ and a rational mapping
\[
\phi: \quad U \longrightarrow  V_\P(g^h)
\]
such that $\phi(U)$ is dense in $V_\P(g^h)$.
Thus, $f\in P_d(g)$ if and only if $f^h(\phi)$ is nonnegative on $U$.
Consequently, the boundary of $P_d(g)$ can be investigated
through studying forms nonnegative on $U$.
We illustrate how to do this as below.

\begin{exm}
Consider the variety $V(g) \subset \cpx^3$ where
\[
g(x) \quad = \quad \left((x_1-1)^2+x_2^2-1\right)^3-x_3^5.
\]
Both $V(g)$ and $V_\P(g^h)$ have positive dimensional singular locus. Let
\[
U=\{ y \in \P^3: y_1^6+y_2^6-y_0y_3^5-y_0^6=0\}.
\]
It is a smooth variety. Let $\phi$ be the mapping:
\[
\phi: \quad \tilde{y}=(y_0,y_1,y_2,y_3) \quad \longmapsto \quad
\tilde{x}=(y_0^3, y_0^3+y_1^3, y_2^3, y_3^3).
\]
Then $\phi(U)=V_\P(g^h)$. So, $f(x)\in P_d(g)$ if and only if $f^h(\phi) \in P_{3d}(q)$,
and $f(x)\in \pt P_d(g)$ if and only if $f^h(\phi) \in \pt P_{3d}(q)$.
Here $q=y_1^6+y_2^6-y_0y_3^5-y_0^6$.
\qed
\end{exm}

However, we would like to remark that such $\phi$ and $U$
are typically quite difficult to find. This issue is beyond the scope of this paper.

\section{Polynomials nonnegative on a semialgebraic set} \label{sec:semi-alg}
\setcounter{equation}{0}

This section studies the cone $P_d(K)$ when $K$ is a general semialgebraic set in $\re^n$.
Consider $K$ is given as
\[
K=\{x \in \re^n: \, (g_1(x), \ldots, g_m(x)) =  0, (p_1(x),\ldots,p_t(x)) \geq 0\}.
\]
Here the $g_i$ and $p_j$ are all polynomials in $x$. Recall that
\[
P_d(K) = \{ f \in \re[x]_{\leq d}: f(x) \geq 0 \quad \forall \, x \in K \}.
\]
We are interested in the algebraic geometric properties of its boundary $\pt P_d(K)$.
Typically, it is a union of hypersurfaces.

We begin with the characterization of the boundary $\pt P_d(K)$.
Like the case of $K$ being a real algebraic variety,
a polynomial positive on $K$ may not lie in the interior of $P_d(K)$.
Let $K^h$ be the projectivization of $K$ which is defined as
\[
K^h=\left\{\tilde{x} \in \re^{n+1}: \,
\big( g_1^h(\tilde{x}), \ldots, g_m^h(\tilde{x}) \big) = 0,
\big( p_1^h(\tilde{x}), \ldots, p_t^h(\tilde{x}) \big) \geq 0\right\}.
\]
Define two constants
\begin{align}
\delta_K(f) & \quad = \quad \min_{ x\in K} \quad  f(x),   \label{df:dtK} \\
\delta_K^h(f) & \quad = \quad \min_{\tilde{x} \in K^h: \|\tilde{x}\|_2=1, x_0 \geq 0}
\quad  f^h(\tilde{x}).  \label{dt-Kh(f)}
\end{align}
Similarly, we say $K^h$ is {\it closed at $\infty$ } if
\[
K^h \cap \{x_0 \geq 0 \} = \mbox{closure}
\left( K^h \cap \{x_0 > 0 \} \right).
\]
We would like to remark that the definitions of $K^h$ and $\delta_K^h(f)$
depend on the defining polynomials of $K$ that are usually not unique.
So in the places where $K^h$ or $\delta_K^h(f)$ appears,
we usually assume the defining polynomials of $K$ are clear from the context.

The interior and boundary of the cone $P_d(K)$
are characterized in the proposition below,
whose proof is almost the same as for Proposition~\ref{char-bdry:Pd(g)}.

\begin{prop}  \label{char:bdrPd(K)}
Let $K$ be given as above.
\bit

\item [(i)] If $K$ is compact, then
\[
\delta_K(f) > 0   \, \Leftrightarrow \, f \in int\big(P_d(K)\big),
\quad \mbox{ and } \quad
\delta_K(f) = 0  \, \Leftrightarrow \, f \in  \pt P_d(K).
\]

\item [(ii)] If $K^h$ is closed at $\infty$, then
\[
\delta_K^h(f) > 0   \, \Leftrightarrow \, f \in int\big(P_d(K)\big),
\quad \mbox{ and } \quad
\delta_K^h(f) = 0  \, \Leftrightarrow \, f \in  \pt P_d(K).
\]
\eit
\end{prop}

Using the above characterization, we can get the following result about $\pt P_d(K)$.

\begin{theorem} \label{thm:bdry-P(K)}
Let $K$ be given as above. Assume
at most $n-m$ inequality constraints are active at any nonzero point in $K^h$.
If either $K$ is compact or $K^h$ is closed at $\infty$,
then the boundary $\pt P_d(K)$ lies on the hypersurface
\[
\mc{E}_d(K):= \left\{ f \in \re[x]_{\leq d}: \,
\prod_{  \{i_1,\ldots, i_k\}  \subseteq [t], k\leq n-m }
\Delta(f, g_1, \ldots,  g_m, p_{i_1},  \ldots,  p_{i_k}) = 0
\right\}.
\]
\end{theorem}

\begin{proof}
Let $f(x) \in \pt P_d(K)$.
First assume $K^h$ is closed at infinity.
So there exists $0 \ne u\in K^h$ such that $f^h(u)=0$.
Let $\{i_1,\ldots, i_k\}$ be the index set of active inequality constraints
\[
p_{i_1}^h(u) = \cdots = p_{i_k}^h(u) = 0.
\]
By assumption, $k\leq n-m$.
Note that $u$ is a minimizer of $f^h$ on $K^h$.
By Fritz-John optimality condition (see Sec. 3.3.5 in \cite{Bsks}),
there exists $(\mu_0, \mu_1, \ldots, \mu_{m+k})\ne 0$ satisfying
\[
\baray{c}
\mu_0 \nabla_{\tilde{x}} f^h(u) +
\overset{m}{\underset{i=1}{\sum}} \mu_i \nabla_{\tilde{x}} g_i^h(u) +
\overset{k}{\underset{j=1}{\sum}} \mu_{m+j} \nabla_{\tilde{x}} p_{i_j}^h(u) = 0, \\
f^h(u) = g_1^h(u) = \cdots = g_m^h(u)=p_{i_1}^h(u)=\cdots=p_{i_k}^h(u)=0.
\earay
\]
So $u$ is a singular solution to the polynomial system
\[
f^h(\tilde{x})= g_1^h(\tilde{x}) =  \cdots = g_m^h(\tilde{x})=
p_{i_1}^h(\tilde{x})= \cdots = p_{i_k}^h(\tilde{x})=0.
\]
Hence, $\Delta(f, g_1,,  \ldots,  g_m, p_{i_1}, \ldots, p_{i_k}) =0$.

The proof is similar when $K$ is compact.
\end{proof}

\begin{figure}[htb]
\centering
\btab{cc}
\includegraphics[width=.45\textwidth]{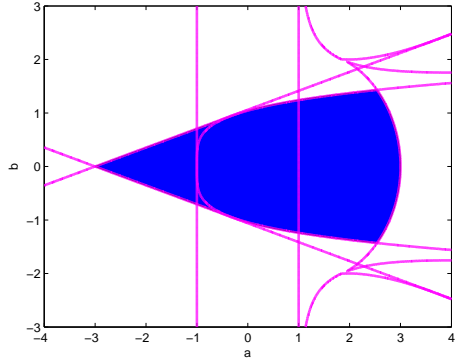}
\etab
\caption{ The picture of $\varphi(a,b)=0$
and the set $F$ in Example~\ref{ball-emp:v2d4}.}
\label{fig-ball:vr2dg4}
\end{figure}

\begin{exm} \label{ball-emp:v2d4}
Consider the polynomials parameterized as
\[
f_{a,b}(x)= x_1^4+x_2^4+a(x_1^3x_2+x_1x_2^3)+b(x_1+x_2)+1,
\]
and $K=\{1-x_1^2-x_2^2 \geq 0\}$ is a ball. From Theorem~\ref{thm:bdry-P(K)},
the boundary of $P_4(K)$ lies on the union of
$\Delta(f_{a,b})=0$ and $\Delta(f_{a,b},g)=0$.
The discriminant $q(a,b) = \Delta(f_{a,b})$ is
\[
2097152(a+1)^2(a-1)^3(a^2+8)^4(32+32a-27b^4)(256+32a^2+27b^4-27ab^4)^2.
\]
By the method used in subsection~\ref{subsec:dis-fg1m},
eliminating $(x,\lmd)$ in \reff{cond:KKT} gives $h(a,b)=0$ where $h(a,b)$ is
\[
\baray{c}
(a+2\sqrt{2}b+3)\cdot (a-2\sqrt{2}b+3)\cdot (a^5+a^3b^2-3a^4 \\
-30a^2b^2-27b^4+32a^3+48ab^2-96a^2+224b^2+256a-768).
\earay
\]
The curve $\Delta(f_{a,b},g)=0$ lies on $h(a,b)=0$.
Let $\varphi(a,b) = h(a,b) \cdot q(a,b)$.
The curves in Figure~\ref{fig-ball:vr2dg4} are defined by $\varphi(a,b)=0$. Let
\[
F = \left\{(a,b) \in \re^2:
\baray{c}
f_{a,b}(x) = \sig_0(x) + \sig_1(x) (1-\|x\|_2^2) \\
\sig_0(x), \sig_1(x) \mbox{ are SOS in } x  \\
\deg(\sig_0) = 4, \quad \deg(\sig_1) = 2
\earay
\right\}.
\]
It is clearly a convex set.
By the method used in Example~\ref{4exmp:psd-form},
$F$ is drawn in the shaded area of Figure~\ref{fig-ball:vr2dg4}.
Let $G=\{(a,b): f_{a,b} \in P_4(K)$. Clearly, $F \subset G$ and the boundary of $G$ lies on $\varphi(a,b)=0$.
Since the polynomials $f_{2,1.5},f_{2,-1.5},f_{4,0}$ are not nonnegative on the unit ball
(verified by {\it GloptiPoly~3} \cite{GloPol3}),
we know $(2,1.5), (2,-1.5), (4,0) \not\in G$.
From Figure~\ref{fig-ball:vr2dg4}, we can observe that $F$ is a maximal convex region
that excludes the pairs $(2,1.5), (2,-1.5), (4,0)$ and
has the boundary lying on $\varphi(a,b)=0$. So $F=G$.
\qed
\end{exm}

Now we discuss the barriers for $P_d(K)$.
The following is similar to Theorem~\ref{psd:nobarfun}.

\begin{theorem} \label{no-bar:intK}
If $K$ has nonempty interior, $d>2$ is even and $n\geq 1$,
then there is no polynomial $\varphi(f)$ satisfying
\bit

\item $\varphi(f)>0$ whenever $f$ lies in the interior of $P_d(K)$, and

\item $\varphi(f)=0$ whenever $f$ lies on the boundary of $P_d(K)$.

\eit
So, $-\log \varphi(f)$ can not be a barrier function for the cone
$P_d(K)$ when we require $\varphi(f)$ to be polynomial in $f$,
and $P_d(K)$ is not representable by LMI.
\end{theorem}
\begin{proof}
Prove by contradiction. Suppose such a $\varphi$ exists.
Since $int(K) \ne \emptyset$, one piece of the boundary $\pt P_d(K)$ must lie
on the irreducible discriminantal hypersurface $\Delta(f)=0$.
The rest of the proof is then almost the same as for Theorem~\ref{psd:nobarfun},
and is omitted here.
\end{proof}

Typically there is no log-polynomial type barrier for the cone $P_d(K)$.
However, $P_d(K)$ has log-semialgebraic type barriers.
When $K$ is compact, $- \log \, \delta_K(f)$, or when $K^h$ is closed at $\infty$,
$- \log \, \delta_K^h(f)$, is a convex barrier for $P_d(K)$,
because both $\delta_K(f)$ and $\delta_K^h(f)$ are semialgebraic,
positive in $int(P_d(K))$, zero on $\pt P_d(K)$, and concave in $f$.
Generally, it is quite difficult to compute $\delta_K(f)$ or $\delta_K^h(f)$
for general $f$ and $K$.
So these two barriers are not very useful in practice.

\subsection{Co-positive polynomials and matrices}

A form $f(x)$ is said to be co-positive if $f(x) \geq 0$ for every $x \in \re_+^n$.
Clearly, $f(x)$ is co-positive if and only if its associated even form
\[
q_f(x) \quad = \quad f(x_1^2,\ldots,x_n^2)
\]
is nonnegative in $\re^{n}$.
A symmetric matrix $A$ is called co-positive if the associated
quadratic form $f(x)=x^TAx$ is co-positive.

Let $\mc{C}_{n,d}$ be the cone of copositive forms in $\re[x]_d$,
and $\pt \mc{C}_{n,d}$ be its boundary.
Clearly, if $f \in \pt \mc{C}_{n,d}$, then there exists $0 \ne u\in \re_+^n$ such that $f(u)=0$,
or equivalently $q_f(\sqrt{u})=0$.
Thus $\pt \mc{C}_{n,d}$ lies on the discriminantal hypersurface $\Delta(q_f)= 0$.

\begin{prop}  \label{pro:bdry-copos}
The Zariski closure of $\pt \mc{C}_{n,d}$ is the hypersurface
\[
\mc{E}_d(\re_+^n):= \left\{f \in \re[x]_d: \,
\prod_{ \emptyset \ne I \subseteq [n] } \Delta (f_I(x_I)) = 0 \right\}.
\]
Here $x_I = (x_i: i\in I)$ and $f_I$ is obtained from $f(x)$ by setting $x_j=0$ for $j\not\in I$.
\end{prop}

\begin{proof}
Let $f\in \pt \mc{C}_{n,d}$. Then there exists $0 \ne u\in \re_+^n$ such that $f(u)=0$.
The index set $I=\{i: u_i >0\}  \subseteq [n]$ is nonempty,
and $f_I(x_I)$ has a positive critical zero point,
because $\nabla_{x_I} f_I(u_I)=0$. So $\Delta (f_I(x_I)) = 0$.
Hence, we have $Zar(\pt \mc{C}_{n,d}) \subseteq \mc{E}_d(\re_+^n)$.
To prove they are equal, we need to show that
$\Delta (f_I(x_I)) = 0$ lies on $Zar (\mc{C}_{n,d})$ for every
$\emptyset \ne I \subseteq [n]$. Fix such an arbitrary $I$. Let $\hat{f}_I(x_I)$
be a co-positive form which vanishes at $\mathbf{1}_I$ ($\mathbf{1}$ is the vector of all ones).
Then there is a neighborhood $\mc{U}$ of $\hat{f}_I$ such that
every $g_I \in \mc{U} \cap \mc{C}_{|I|,d}$ vanishes somewhere near $\mathbf{1}_I$.
Thus $\mc{U} \cap \mc{C}_{|I|,d} \subset \{\Delta (f_I(x_I)) = 0\}$, and
\[
Zar(\mc{U} \cap \mc{C}_{|I|,d}) \subseteq Zar(\{\Delta (f_I(x_I)) = 0\})= \{\Delta (f_I(x_I)) = 0\}.
\]
Since the hypersurface $\Delta (f_I(x_I)) = 0$ is irreducible, we must have
\[
\{\Delta (f_I(x_I)) = 0\} \subseteq Zar (\mc{U} \cap \mc{C}_{|I|,d}) \subseteq Zar (\mc{C}_{n,d}).
\]
The above is true for every $\emptyset \ne I \subset [n]$.
So $Zar (\mc{C}_{n,d}) = \mc{E}_d(\re_+^n)$.
\end{proof}

Proposition~\ref{pro:bdry-copos} is equivalent to the fact that
\[
\Delta(q_f)= 0  \, \Longleftrightarrow \,   \prod_{ \emptyset \ne I \subseteq [n] } \Delta (f_I) =0.
\]
This is because $\nabla_x (q_f(x)) = 2 \diag(x) \cdot \nabla_x f(x^2)$ and
\begin{align*}
\Delta(q_f) = 0
& \Longleftrightarrow   Res\left(x_1 \frac{\pt f}{\pt x_1}(x^2), \ldots, x_n\frac{\pt f}{\pt x_n}(x^2)\right) =0 \\
& \Longleftrightarrow   Res\left(x_1 \frac{\pt f}{\pt x_1}(x), \ldots, x_n\frac{\pt f}{\pt x_n}(x)\right) =0 \\
& \Longleftrightarrow  \prod_{ \emptyset \ne I \subseteq [n] } \Delta(f_I) =0.
\end{align*}
We refer to Theorem~1.2 in \cite[Chapt.10]{GKZ} for the last equivalence in the above.
If $[n]\backslash I = \{i_1,\ldots,i_k\}$, \reff{copoly:dis} implies
$\Delta(f_I(x_I)) = \eta \Delta(f,x_{i_1},\ldots,x_{i_k})$ for some $\eta \ne 0$.
In particular, if $d=2$ and $f(x)=x^TAx$ is quadratic,
then Proposition~\ref{pro:bdry-copos} and \reff{quad:dis=det}
imply $Zar(\pt \mc{C}_{n,2})$ is the hypersurface
\be \label{defpoly:copos}
\prod_{ \emptyset \ne I \subseteq [n] } \det A(I,I) = 0.
\ee

\begin{cor} \label{no-bar:copos}
Suppose $d\geq 2$ and $n\geq 2$. Then there is no polynomial $\varphi(f)$ satisfying
\bit

\item $\varphi(f)>0$ whenever $f$ is in the interior of $\mc{C}_{n,d}$, and

\item $\varphi(f)=0$ whenever $f$ is on the boundary of $\mc{C}_{n,d}$.

\eit
So, $-\log \varphi(f)$ can not be a barrier function for the cone
$\mc{C}_{n,d}$ when we require $\varphi(f)$ to be polynomial in $f$,
and $\mc{C}_{n,d}$ is not representable by LMI.
\end{cor}
\begin{proof}
Prove the first part by contradiction. Suppose such a $\varphi(f)$ exists. Then
\[
\varphi(f) = 0 \quad \forall f \in \pt \mc{C}_{n,d}.
\]
So the Zariski closure of $\pt \mc{C}_{n,d}$ lies on the hypersurface $\varphi(f) = 0$.
Since $d\geq 2$, $\Delta(f)$ is an irreducible polynomial in $f$.
By Proposition~\ref{pro:bdry-copos}, the hypersurface $\Delta(f)=0$
lies on $\varphi(f)=0$, and $\varphi(f)$ vanishes on $\Delta(f)=0$.
By Hilbert Nullstellensatz (see Theorem\ref{HilNull}), there exist a positive integer $k>0$
and a polynomial $\phi(f)$ such that
\[
\varphi(f)^k = \phi(f) \Delta(f).
\]
In particular, if we choose $f$ to be $\hat{f}(x) =(\mathbf{1}_{n}^Tx)^d$ in the above, then
\[
\varphi(\hat{f})^k = \phi(\hat{f}) \Delta(\hat{f}) = 0.
\]
This is because the form $\hat{f}(x)$ has a nonzero critical point when $d\geq 2$ and $n\geq 2$.
However, $\hat{f}(x)$ clearly lies in the interior of $\mc{C}_{n,d}$,
which contradicts the first item.

The second part clearly follows the first part.
\end{proof}

\noindent
{\it Remark:} Corollary~\ref{no-bar:copos} would be implied by Theorem~\ref{no-bar:intK}
for the case that $d>2$ is even.

\begin{exm} \label{2exmp:copos}
(i) Consider the symmetric matrices $A$ parameterized as
\[
A =
\left[
\baray{rrrr}
1 & a & -b & b \\
 a & 1 & -b & -a \\
-b & -b & 1 & -a \\
 b & -a & -a & 1
\earay\right].
\]
We are interested in the set of all pairs $(a,b)$
such that $A$ is co-positive.
The polynomial $\varphi(a,b)$ defining equation \reff{defpoly:copos} is
\[
\baray{c}
- {\left(a - 1\right)}^5\cdot {\left(a + 1\right)}^3\cdot {\left(b - 1\right)}^3\cdot
 {\left(b + 1\right)}^5\cdot {\left( - 2 b^2 + a + 1\right)}^2\cdot
 {\left(2a^2 + b - 1\right)}^2 \cdot \\
  \left(a^2 + 3 a b + a + b^2 - b - 1\right)\cdot
  \left( -a^2 + a b + a - b^2 - b + 1\right).
\earay
\]
The curve $\varphi(a,b)=0$ is drawn in the left picture of Figure~\ref{fig:copos}. Let
\[
F = \left\{(a,b) \in \re^2:  A = X + Y, X \succeq 0, Y \geq 0 \right\}.
\]
By the method used in Example~\ref{4exmp:psd-form},
$F$ is drawn in the shaded area of
the left picture in Figure~\ref{fig:copos}.
Because every co-positive $4\times 4$ matrix is a sum
of a nonnegative matrix and a positive semidefinite matrix (see \cite{Dian}),
we know $F=\{(a,b): A \in \mc{C}_{4,2}\}$.

\begin{figure}[htb]
\centering
\btab{cc}
\includegraphics[height=.3\textwidth]{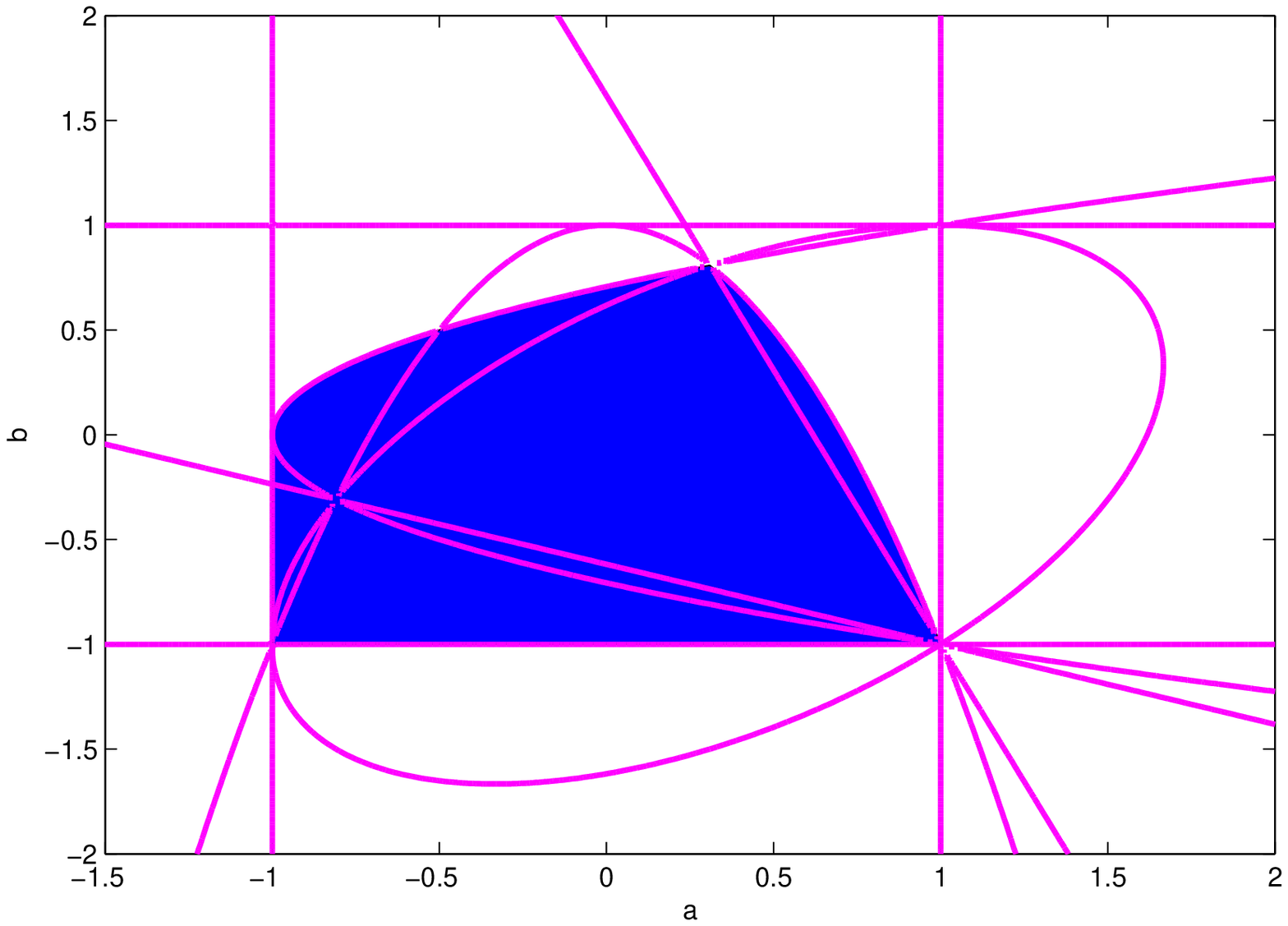} &
\includegraphics[height=.3\textwidth]{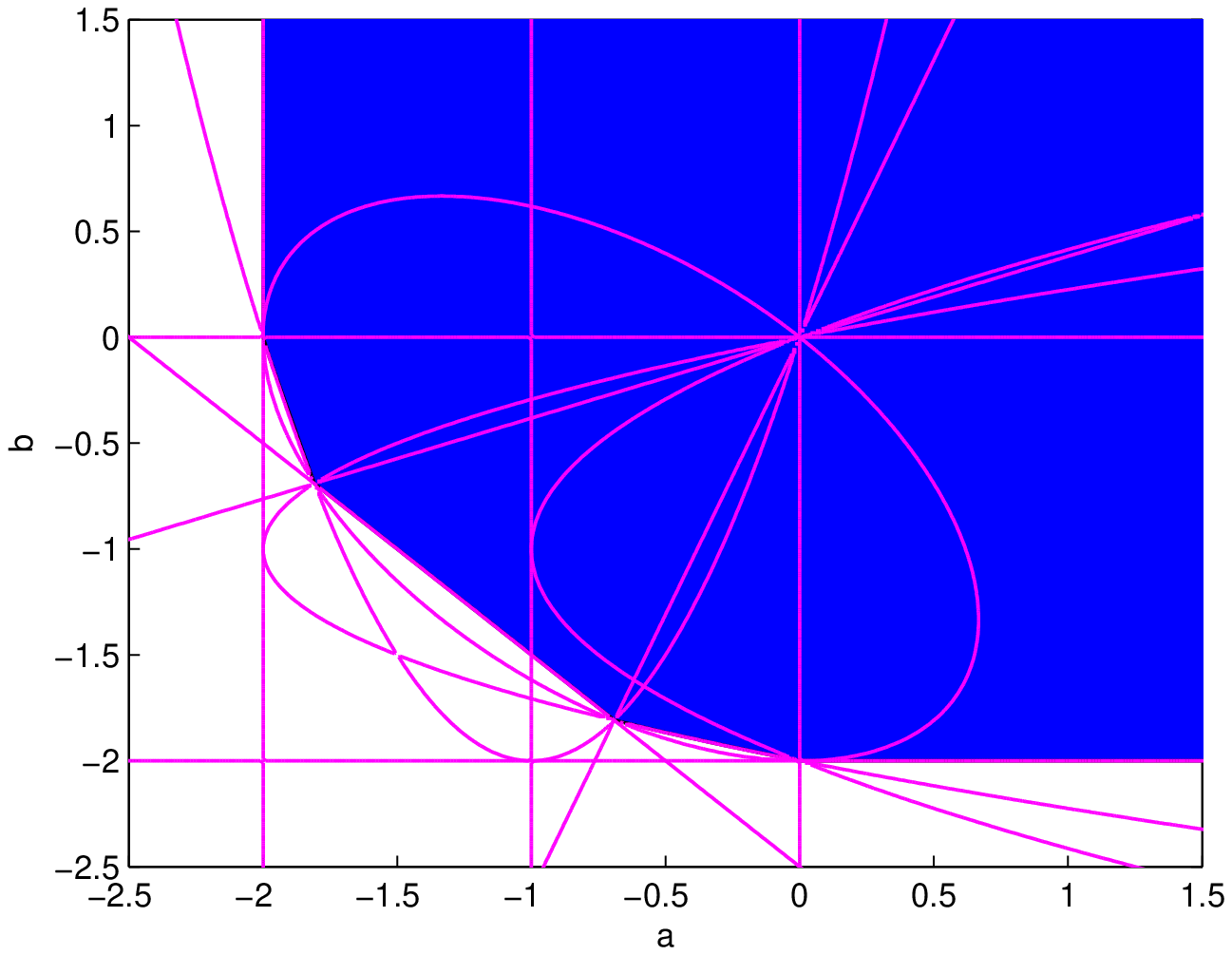}
\etab
\caption{ The pictures of the curve $\varphi(a,b)=0$
and region $F$ for co-positive matrices in Example~\ref{2exmp:copos}.
The left is for (i), and the right for (ii).}
\label{fig:copos}
\end{figure}

\smallskip
\noindent
(ii) Consider the symmetric matrices $A$ parameterized as
\[
A =
\left[\baray{lllll}
1   & 1+a  & 1+b & 1+b & 1+a \\
1+a & 1    & 1+a & 1+b & 1+b \\
1+b & 1+a  & 1   & 1+a & 1+b \\
1+b & 1+b  & 1+a & 1   & 1+a \\
1+a & 1+b  & 1+b & 1+a &  1 \\
\earay\right].
\]
When $a=-2,b=0$, it is the matrix associated to the Horn's copositive form (see Reznick \cite{Rez00}).
The polynomial $\varphi(a,b)$ defining equation \reff{defpoly:copos} is
\[
\baray{c}
a^9\cdot b^{10}\cdot (a + 1)\cdot {(a + 2)}^4\cdot {(b + 2)}^4\cdot
 {(2a^2 + 4a - b)}^5{(2b^2 + 4b - a)}^5\cdot \\
 {(a^2 - 3ab + b^2)}^7\cdot
 (b^2 + 2b - a)(2a + 2b + 5)\cdot
 {(a^2 + ab + 2a + b^2 + 2b)}^5.
\earay
\]
The curves in the right picture of Figure~\ref{fig:copos} are defined by $\varphi(a,b)=0$. Let
\[
F = \left\{(a,b) \in \re^2:  \|x\|_2^2 \cdot
\left(\sum_{1\leq i,j\leq 5} A_{i,j}x_i^2x_j^2\right) \mbox{ is SOS in } x \right\}.
\]
It is an unbounded convex set. By the method used in Example~\ref{4exmp:psd-form},
$F$ is drawn in the shaded area of the right picture in Figure~\ref{fig:copos}.
Let $G=\{(a,b): A \in \mc{C}_{5,2}\}$.
Clearly, $F\subset G$ and the boundary of $G$ lies on $\varphi(a,b)=0$.
Then $f_{a,b}(1,1,0,0,0) \geq 0$,
$f_{a,b}(1,0,1,0,0) \geq 0$, and $f_{a,b}(1,1,1,1,1) \geq 0$ imply that
any pair $(a,b) \in G$ satisfies
\[
2a+2b+5 \geq 0, \quad a+2 \geq 0, \quad b+2 \geq 0.
\]
Since $(-0.5,-1.88),(-1.88,-0.5), (-1.3,-1.3) \not\in G$
(verified by {\it GloptiPoly~3} \cite{GloPol3}),
from the right picture in Figure~\ref{fig:copos},
we can observe that $F$ is a maximal convex region that
satisfies the above three linear constraints, excludes
the previous $3$ pairs and has the boundary lying on $\varphi(a,b)=0$.
So $F=G$.
\qed
\end{exm}

\section{Conclusions and discussions}  \label{sec:con-dis}
\setcounter{equation}{0}

This paper studies the algebraic geometric properties of the boundary $\pt P_d(K)$.
When $K=\re^n$, $\pt P_d(K)$ lies on an irreducible hypersurface defined
by the discriminant of a single polynomial;
when $K$ is a real algebraic variety, the boundary $\pt P_d(K)$ lies on a hypersurface defined
by the discriminant of several polynomials;
when $K$ is a general semialgebraic set, the boundary $\pt P_d(K)$ lies on
a union of discriminantal hypersurfaces.
General degree formulae for these hypersurfaces and discriminants are also proved.
An interesting consequence of these results is that
$-\log \varphi(f)$ can not be a barrier for the cone $P_d(K)$
when $\varphi(f)$ is required to be polynomial in $f$,
but it would be a barrier if $\varphi(f)$ is allowed to be semialgebraic.

Given general multivariate polynomials $f_0,\ldots,f_m$,
how to compute the discriminant of type $\Delta(f_0,\ldots,f_m)$?
When $m=0$, there are standard procedures for computing $\Delta(f_0)$.
However, to the best of the author's knowledge, this question is open for $m>0$.
In computing $\Delta(f)$ for a single polynomial $f$,
it is typically non-practical to get a general formula for $\Delta(f)$,
but if $f(x)$ has a few terms and its coefficients have a few parameters,
is there any practical method for evaluating $\Delta(f)$ efficiently?
These questions are interesting future work.

%
%

\bigskip\noindent
{\bf Acknowledgement} \,
The author would like very much to thank Bill Helton, Kristian Ranestad,
Jim Renegar and Bernd Sturmfels
for fruitful suggestions on improving this paper.

\end{document}